\documentclass{amsart}
\usepackage{amssymb}
\newtheorem{theorem}{Theorem}[section]
\newtheorem{definition}[theorem]{Definition}
\newtheorem{corollary}[theorem]{Corollary}
\newtheorem{lemma}[theorem]{Lemma}
\newtheorem{proposition}[theorem]{Proposition}
\theoremstyle{definition}
\theoremstyle{remark}
\newtheorem{rem}[theorem]{Remark}
\numberwithin{equation}{section}
\newcommand{\ty}{\mathbf t}

\newcommand{\Qp}{{\mathbb Q}_p}
\newcommand{\Q}{\mathbb Q}
\newcommand{\N}{\mathbb N}
\newcommand{\F}{\mathbb F}
\newcommand{\R}{\mathbb R}

\newcommand{\HH}{\mathcal H}
\newcommand{\Fp}{{\mathbb F}_p}
\newcommand{\Z}{\mathbb Z}
\newcommand{\ZK}{{\mathbb Z}_K}
\newcommand{\p}{\mathfrak p}
\renewcommand{\l}{\mathfrak l}
\newcommand{\q}{\mathfrak q}
\newcommand{\st}[1]{\vskip 1mm\noindent{\bf #1}}
\newcommand{\stst}[1]{\vskip 1mm\noindent\hskip4mm{\bf #1}}
\newcommand{\ststst}[1]{\vskip 1mm\noindent\hskip8mm{\bf #1}}
\def\cpl{\operatorname{compl}}
\def\disc{\operatorname{disc}}
\def\md#1{\ \mbox{\rm(mod }{#1})}
\def\gb#1{\overline{\gamma_{#1}(\t)}}
\def\ind{\operatorname{ind}}
\def\j{\mathbf{j}}
\def\l{\mathfrak{l}}
\def\la{\lambda}
\def\n{\mathbf{n}}

\def\opt{\operatorname{opt}}
\def\om{\omega}
\def\ord{\operatorname{ord}}

\def\t{\theta}
\def\T{\mathbf{T}}
\def\tq{\,\,|\,\,}
\def\zpx{\Z_p[x]}
\begin{document}
\title[Higher Newton polygons, discriminants and prime ideal decomposition]{Higher Newton polygons in the computation of discriminants and prime ideal decomposition \\in number fields}
\author[Gu\`ardia]{Jordi Gu\`ardia}
\address{Departament de Matem\`atica Aplicada IV, Escola Polit\`ecnica Superior d'Enginyera de Vilanova i la Geltr\'u, Av. V\'\i ctor Balaguer s/n. E-08800 Vilanova i la Geltr\'u, Catalonia}
\email{guardia@ma4.upc.edu}

\author[Montes]{\hbox{Jes\'us Montes}}
\address{Departament de Ci\`encies Econ\`omiques i Socials,
Facultat de Ci\`encies Socials,
Universitat Abat Oliba CEU,
Bellesguard 30, E-08022 Barcelona, Catalonia, Spain}
\email{montes3@uao.es}

\author[Nart]{\hbox{Enric Nart}}
\address{Departament de Matem\`{a}tiques,
         Universitat Aut\`{o}noma de Barcelona,
         Edifici C, E-08193 Bellaterra, Barcelona, Catalonia}
\email{nart@mat.uab.cat}
\thanks{Partially supported by MTM2006-15038-C02-02 and MTM2006-11391 from the Spanish MEC}
\date{}
\keywords{number field, Newton polygon, prime ideal decomposition, discriminant, index}

\subjclass[2000]{Primary 11R04; Secondary 11R29, 11Y40}

\maketitle

\begin{abstract}
We present an algorithm for computing discriminants and prime ideal decomposition in number fields. The algorithm is a refinement of a $p$-adic factorization method based on Newton polygons of higher order. The running-time
and memory requirements of the algorithm appear to be very good: for a given prime number $p$, it computes the $p$-valuation of the discriminant and the factorization of $p$ in a number field of degree $1000$ in a few seconds, in a personal computer.
\end{abstract}


\section{Introduction}

The factorization of prime numbers in number fields is a classical problem, whose resolution lays at the foundation of algebraic number theory. Although it is completely understood from the theoretical point of view, the rising of computational number theory in the last decades has renewed the interest on the problem from a practical perspective. In his comprehensive book \cite[p. 214]{Cohen}, H. Cohen  refers to this problem as one of the main computational tasks in algebraic number theory.

The most common insight in the known solutions of the problem is based on the solution on a more general problem: the determination of a (local) integral basis. There is a number of highly efficient methods for this problem, due to H. Zassenhaus and M. Pohst \cite{PZ}, D. Ford and P. Letard \cite{Ford}, and D. Ford, S. Pauli and X. Roblot \cite{FPR}.

The theory of higher order Newton polygons developed in \cite{montes} and revised in \cite{GMN}, HN standing for ``higher Newton", has revealed itself as a powerful tool for the analysis of the decomposition of a prime $p$ in a number field. Higher Newton polygons are a $p$-adic tool, and their computation involves no extension of the ground field, but only extensions of the residue field; thus, they constitute an excellent tool for a computational treatment of the problem. In this paper we explain how
the theoretical results of \cite{GMN} apply  to  yield an algorithm, due to the second author \cite[Ch.3]{montes}, to factor a prime number $p$ in a number field $K$, in terms of a generating equation $f(x)$. The algorithm computes the $p$-valuation of the index of $f(x)$ as well; in particular, it determines the discriminant of the number field, once one is able to factorize the discriminant $\disc(f)$ of the defining equation.

In many applications, the computation of an integral basis is very useful because it helps to carry out other tasks in the number field. However, if one is interested only in the discriminant or in the factorization of a prime, our direct method has the advantage of being more efficient and it makes possible to carry out these tasks in number fields of much higher degree. In fact, the running-time and memory requirements of the algorithm appear to be very good. Even in some bad cases, chosen to test the limit of its capabilities, it computes the factorization of $p$ in a number field of degree $1000$ and $p$-index $200000$ in a few seconds, in a personal computer. If we add the computation of generators of the prime ideals, the running-time may increase in a significant way, because this routine implies an extended gcd computation.

The outline of the paper is as follows. In section 2 we present the basic algorithm that is obtained by a direct application of the ideas of \cite{GMN}. In section 3 we introduce an optimization based on a lowering of the order in which the computations take place, and we prove a strong optimization result (Theorem \ref{optimal}). We refer to this optimization process as  \emph{refinement}, and it results in a dramatic lowering of the complexity. In section 4 we show how to compute generators of the prime ideals lying above $p$ in terms of the output of the algorithm. In section 5 we describe an implementation, and in section 6 we present the results of some numerical tests. We construct some ``worse possible" polynomials, that should be specially difficult with respect to the structure of the algorithm; this means that they have a huge index, and this index is sufficiently ``hidden" to force the algorithm to work in a high order. The  record is a polynomial of degree $6912$ and $2$-index $77673504$, for which the factorization of $2$ is obtained in $787$ seconds. The algorithm, moreover, is highly parallelizable, so that it can raise the bounds of computations on number fields to huge degrees.

The local nature of all the computations involved in the algorithm justifies its high efficiency compared to the classical insight explained above. Anyway, after this algorithm, one  can  go the other way round and apply it as a previous step in the determination of an integral basis. Numerical experimentation suggests that this new approach provides a significant improvement in the solution of this problem.

\section{Computation of discriminants and prime ideal decomposition in number fields}
We fix a number field $K=\Q(\theta)$, generated by a monic irreducible polynomial $f(x)\in\Z[x]$, such that $f(\t)=0$. We denote by  $\ZK$ the ring of integers of $K$. We fix also a prime number $p\in\Z$.
The $p$-adic valuation is denoted simply by $v$ in order to avoid confusion with $p$-adic valuations $v_r$ of higher order.
If $\F$ is a finite field and $\varphi(y),\,\psi(y)\in\F[y]$, we write $\varphi\sim\psi$ to indicate that the two polynomials coincide up to multiplication by a nonzero constant in $\F$.

In this section we present the basic algorithm that computes the $p$-value of the discriminant of $K$ and the prime ideal decomposition of $p\,\ZK$, that is obtained by a direct application of the ideas of \cite{GMN}.

\subsection{Types ands their representatives}\label{seccio-tipus}
The basic tool for the algorithm is the concept of type and its representative, which we recall here with some detail. All results of this section are taken from \cite[\S2]{GMN}.

\begin{definition}\label{definicio-tipus}
A type of order zero is a monic irreducible polynomial in $\F_p[y]$. Let $r\ge 1$ be a natural number. A {\em type of order $r$} is a sequence of data:
$$
\ty=(\phi_1(x);\lambda_1, \phi_2(x);\dots;\lambda_{r-1},  \phi_{r}(x); \lambda_{r},\psi_{r}(y)),
$$
where $\phi_1(x), \dots, \phi_r(x)\in\Z[x]$ are monic  polynomials,
 $\lambda_1,\dots,\lambda_r\in\Q^-$ are ne\-gative rational numbers, and
$\psi_r(y)\in \overline{\F}_p [y]$ is a monic polynomial, that satisfy the following pro\-perties:
\begin{enumerate}
\item $\phi_1(x)$ is irreducible modulo $p$. Let $\psi_0(y)\in\F_p[y]$ be the polynomial obtained by reduction of $\phi_1(y)$ modulo $p$. We define $\F_1:=\F_p[y]/(\psi_0(y))$.
\item For all $1\le i<r$, the Newton polygon of $i$-th order, $N_i(\phi_{i+1})$, is one-sided, with positive length and slope $\la_i$.
\item For all $1\le i<r$, the residual polynomial of $i$-th order with respect to $\la_i$, $R_i(\phi_{i+1})(y)$, is an irreducible polynomial in $\F_i[y]$. Let $\psi_i(y)\in \F_i[y]$ be the monic polynomial determined by $R_i(\phi_{i+1})(y)\sim\psi_i(y)$. We define $\F_{i+1}:=\F_i[y]/(\psi_i(y))$.
\item $\psi_r(y)\in\F_r[y]$ is a monic irreducible polynomial, $\psi_r(y)\ne y$.  We define $\F_{r+1}:=\F_r[y]/(\psi_r(y))$.
\end{enumerate}
\end{definition}

Every type carries implicitly a certain amount of extra data, whose notation we fix now.
For all $1\le i\le r$:
\begin{itemize}
\item $h_i, e_i$ are a pair of positive coprime integers such that $\lambda_i=-h_i/e_i$,
\item $\ell_i,\,\ell'_i\in\Z$ are fixed integers such that $\ \ell_i h_i-\ell'_ie_i=1$,
\item $f_i=\deg \psi_i(y)$, and $f_0=\deg \psi_0(y)=\deg \phi_1(x)$,
\item $m_i=\deg \phi_i(x)$, and $m_{r+1}=m_re_rf_r$. Note that $m_{i+1}=e_if_im_i$,
\item $z_i=y \md{\psi_i(y)}\in\F_{i+1}^*$, $z_0=y \md{\psi_0(y)}\in\F_1^*$. Thus, $\F_{i+1}=\F_i(z_i)$.
\end{itemize}

Also, for all $1\le i\le r+1$, the type carries certain $p$-adic discrete valuations $v_i:\Qp(x)^\ast\to \Z$ \cite[Def.2.5]{GMN}, and
semigroup homomorphisms,
$$\om_i\colon \zpx\setminus\{0\} \to\Z_{\ge 0},\qquad P(x)\mapsto \ord_{\psi_{i-1}}(R_{i-1}(P)),
$$ where $R_0(P)(y)\in\F_p[y]$ is the reduction modulo $p$ of $P(y)/p^{v(P)}$.
These objects play an essential role in what follows, because $\om_i(P)$ measures the length of the principal part, $N_i^-(P)$, of the Newton polygon of $i$-th order of $P(x)$ \cite[Lem.2.17]{GMN}. The principal part $N^-$ of a polygon $N$ is the polygon determined by the sides of negative slope of $N$.

To avoid confusion, in case of working simultaneously with different types, we add a superscript with the type to every component or datum: $\phi_i^\ty(x)$, $\la_i^\ty$, $e_i^\ty$, etc.

\begin{definition}
We say that $\lambda_i,\phi_{i+1}(x)$ (and their implicit data) are the $i$-th {\em level} of $\ty$.
\end{definition}

By truncation we can easily obtain types of lower order. We denote
$$
\ty_i:=(\phi_1(x);\lambda_1,\phi_2(x);\cdots;\lambda_{i-1},\phi_i(x);\lambda_i,\psi_i(y)), \quad 0\le i\le r,
$$
$$
\tilde{\ty}_i:=(\phi_1(x);\lambda_1,\phi_2(x);\cdots;\lambda_{i-1},\phi_i(x);\lambda_i,\phi_{i+1}(x)), \quad 0\le i< r.
$$
Clearly, $\ty_i$ is a type of order $i$. The \emph{extension} $\tilde{\ty}_i$ is not a type, strictly speaking.

To our polynomial $f(x)\in\Z[x]$ we can attach the set $\ty_0(f)$ of all types of order zero that divide $f(x)$ modulo $p$. By Hensel's lemma, each $\ty\in\ty_0(f)$ determines a monic $p$-adic factor $f_\ty(x)\in\Z_p[x]$ of $f(x)$, and
$$
f(x)=\prod_{\ty\in\ty_0(f)}f_\ty(x).
$$
The types of order $r$ play an analogous role and they provide similar factorizations in higher order.
Let us recall some concepts and results in this respect.

\begin{definition}\label{type} Let $\ty$, $\ty'$ be types of order $r$, and let $P(x)\in\zpx$ be a monic polynomial.
\begin{itemize}
 \item We say that $P(x)$ has type $\ty$ if $\deg P=m_{r+1}\om_{r+1}(P)>0$, or equivalently
\begin{enumerate}
 \item $P(x)\equiv \phi_1(x)^{a_0} \md{p}$, for some positive integer $a_0$, and
\item  For all $1\le i\le r$, the Newton polygon $N_i(P)$ is one-sided, of slope $\lambda_i$, and
 $R_i(P)(y)\sim\psi_i(y)^{a_i}$ in $\F_i[y]$, for some positive integer $a_i$.
\end{enumerate}
\item We say that $P(x)$ is divisible by $\ty$, or that $\ty$ divides $P(x)$, if $\om_{r+1}^\ty(P)>0$. Formally, we can think of $\om_{r+1}^\ty(P)$ as the exponent with which $\ty$ divides $P(x)$.

\noindent If $\ty$ divides $P(x)$, we denote by $P_{\ty}(x)$ the monic factor of $P(x)$ of  type $\ty$ and greatest degree. It has $\deg P_{\ty}=m_{r+1}\om_{r+1}^\ty(P)$, and $\om_{r+1}^\ty(P_\ty)=\om_{r+1}^\ty(P)$.
\item We say that $\ty$ and $\ty'$ are $P$-equivalent, if both divide $P(x)$, and $P_\ty(x)=P_{\ty'}(x)$.
\item We say that a set $\T$ of types faithfully represents $P(x)$, if $P(x)$ is divisible by all types in $\T$, and  $P(x)=\prod_{\ty\in\T}P_\ty(x)$.
\end{itemize}
\end{definition}

In \cite[\S2.3]{GMN} it is described a constructive method to enlarge a type of order $r$  into different types of order $r+1$.

\begin{theorem}\label{construeix-phi} Let $\ty$ be a type of order $r$.
We can effectively construct a monic polynomial $\phi_{r+1}(x)\in\Z[x]$ of type $\ty$ such that $\om_{r+1}(\phi_{r+1})=1$. This polynomial has minimal degree $\deg \phi_{r+1}=m_{r+1}$ among all polynomials of type $\ty$.
\end{theorem}

We call such a polynomial $\phi_{r+1}(x)$ a \emph{representative of the type} $\ty$. We denote by $\tilde{\ty}:=(\phi_1(x);\cdots;\la_r,\phi_{r+1}(x))$, the extension of $\ty$ by $\phi_{r+1}(x)$; this object is prepared to be enlarged to a type of order $r+1$, $(\tilde{\ty};\lambda_{r+1},\psi_{r+1}(y))$, simply by taking any negative rational number $\lambda_{r+1}\in\Q^-$ and any irreducible monic polynomial $\psi_{r+1}(y)\in\F_{r+1}[y]$.

The representative of a type plays the analogous role in order $r$ to that played by an irreducible polynomial modulo $p$ in order one.

\subsection{Types versus prime ideals. The Basic algorithm}\label{ideals}
Recall that we have fixed a monic irreducible polynomial $f(x)\in\Z[x]$.

\begin{definition}
A type $\ty$ of order $r$ is said to be {\em $f$-complete} if $\om_{r+1}^\ty(f)=1$.
\end{definition}

\begin{theorem}[{\cite[Cor.3.8]{GMN}}]\label{criteri}
Let $\ty$ be an $f$-complete type of order $r$. Then the $p$-adic factor $f_\ty(x)$ is irreducible in $\Z_p[x]$. Moreover, if $L/\Q_p$ is the extension generated by $f_\ty(x)$, we have
$$
e(L/\Q_p)= e_1 \cdots e_{r},\quad f(L/\Q_p)= f_0f_1 \cdots f_{r}.
$$
\end{theorem}
Thus, an $f$-complete type singles out a unique prime ideal $\p$ dividing $p\,\ZK$, whose ramification index and residual degree can be read in the data of $\ty$:
$$
e(\p/p)= e_1 \cdots e_{r},  \quad f(\p/p)=f_0f_1 \cdots f_{r}.
$$

The $p$-adic factorization process of \cite{GMN} consists essentially in the cons\-truction of a set $\T$ of $f$-complete types,  that faithfully represents $f(x)$. Thus, it can be interpreted as a \emph{Basic algorithm}, to determine the prime ideal decomposition of $p\,\ZK$.
The types are built iteratively by means of  Theorem \ref{construeix-phi}, and the theory of Newton polygons of higher order. We start with the set $\T_0(f):=\ty_0(f)$, that faithfully represents $f(x)$. We extend the non-$f$-complete types of this set to types of order one, in order to construct a set $\T_1(f)$  that, again, faithfully represents $f(x)$, etc.
At each order $r$, the extension process is carried out by a main loop that performs the following operations.\bigskip

\noindent{\bf Main loop of the Basic algorithm. }At the input of a non-$f$-complete type $\ty$ of order $r-1$, for which $\om_r(f)>0$, and a representative $\phi_r(x)$:
\begin{itemize}
\item[1)] Compute the Newton polygon of $r$-th order, $N_r(f)=S_1+\dots+S_t$, with respect to $\ty$ and $\phi_r(x)$.
\item[2)] For every side $S_j$ of negative slope $\lambda_{r,j}<0$, compute the residual polynomial of $r$-th order,
$R_{r,j}(f)(y)\in\F_{r}[y]$,  with respect to $\ty$, $\phi_{r}(x)$ and $\lambda_{r,j}$.
\item[3)] Factorize this polynomial in $\F_{r}[y]$:
$$
R_{r,j}(f)(y)=\psi_{r,1}(y)^{a_1}\cdots\psi_{r,s}(y)^{a_s}.
$$
\item[4)] For every factor $\psi_{r,k}(y)$, compute a representative of the type $\ty^{j,k}:=(\tilde{\ty};\lambda_{r,j},\psi_{r,k}(y))$.
\end{itemize}
For those factors  $\psi_{r,k}(y)$ with exponent $a_k=1$, the type $\ty^{j,k}$ is complete. For the remaining types we continue the iterative process.

Thus, each non-complete type of order $r-1$ has sprouted several types of order $r$, which are called \emph{branches} of the input
type $\ty$. We have a factorization in $\Z_p[x]$:
$$
f_\ty(x)=\prod_{j,k}f_{\ty^{j,k}}(x),
$$
with $\deg f_{\ty^{j,k}}=e_{r,j}f_{r,k}m_{r}$. Also, $(\om_{r+1})^{\ty^{j,k}}(f)>0$, for all $j,k$, and
\begin{equation}\label{distance}
\om_{r}^\ty(f)=\sum_{j,k}e_{r,j}f_{r,k}(\om_{r+1})^{\ty^{j,k}}(f).
\end{equation}
Hence, the invariant $\om_{r}^{\ty}(f)$ is an upper bound for the number of irreducible factors of $f_\ty(x)$, and it is a kind of measure of the distance that is left to complete the analysis of the type $\ty$ and its branches (or equivalently, to decompose each $f_\ty^{j,k}(x)$ into a product of irreducible factors). Also, (\ref{distance}) shows that, except for the case in which there is only one branch with $e_{r}=f_{r}=1$, the branches are always closer to be $f$-complete than $\ty$.

We denote by $\ty_{r}(f)$ the set of types of order $r$ obtained by aplying the Main loop to all non-$f$-complete types of $\ty_{r-1}(f)$. We denote by $\ty_i(f)^{\cpl}$ the subset of the $f$-complete types of $\ty_i(f)$, and we define
$$
\T_{r}(f):=\ty_{r}(f)\cup\left(\bigcup_{0\le i<r}\ty_i(f)^{\cpl}\right).
$$
\begin{proposition}[{\cite[\S3]{GMN}}]\label{end}
$\T_{r}(f)$ faithfully represents $f(x)$.
\end{proposition}

To show that the Basic algorithm deserves this name, we have to prove that, after a finite number of enlargements, all types of $\ty_r(f)$ will be complete. To this purpose we introduce another variable to measure how far is a type from being complete, that works even in the unibranch case with $e_{r}=f_{r}=1$. This control variable is defined in terms of \emph{higher indices}.

\subsection{Indices of higher order}\label{secthindex}
The results of this section are extracted from \cite[\S4]{GMN}. Denote
$$
\ind(f):=v\left(\left(\ZK\colon\Z[\t]\right)\right),
$$
and recall the well-known relationship, $v(\disc(f))=2\ind(f)+v(\disc(K))$, between $\ind(f)$, the discriminant of $f(x)$
and the discriminant of $K$.

\begin{definition}
Let $N=S_1+\cdots+S_t$ be a principal polygon, with finite sides ordered by increasing slopes $\la_1<\cdots<\la_t<0$. Denote by $E_i=\ell(S_i)$, $H_i=H(S_i)$, $d_i=d(S_i)$ the respective length, height and degree of each side \cite[\S1.1]{GMN}. We define the \emph{index of the polygon} $N$ to be the nonnegative integer
$$
\ind(N):=\sum_{i=1}^t\frac12(E_iH_i-E_i-H_i+d_i)+\sum_{1\le i<j\le t}E_iH_j.
$$
\end{definition}
This number is equal to the number of points with integral coordinates that lie below or on the polygon, strictly above the horizontal line that passes through the last point of $N$ and strictly beyond the vertical axis. Hence, $\ind(N)=0$ if and only if $N$ has a unique side with height $H=1$, or length $E=1$.

\begin{definition}\label{index}
Let $\ty$ be a type of order $r-1$, and let $\phi_{r}(x)$ be a representative of $\ty$. We define its $f$-{\em index} to be the nonnegative integer $$\ind_\ty(f):=\ind_{\ty,\phi_r}(f):=f_0\cdots f_{r-1}\ind(N_r^-(f)),$$ the Newton polygon of $r$-th order taken with respect to $\ty$ and $\phi_r(x)$.

We say that $\ty$ is \emph{$f$-maximal} if $\ty$ divides $f(x)$ and $\ind_\ty(f)=0$.

For any natural number $r\ge1$, we define $\ind_r(f):=\sum_{\ty\in\ty_{r-1}(f)}\ind_\ty(f).$
\end{definition}

Since the Newton polygon $N_{r}^-(f)$ depends on the choice of $\phi_{r}(x)$, the value of $\ind_\ty(f)$, and the fact of being $f$-maximal, depends on this choice too.

\begin{proposition}[{\cite[Lem.4.16]{GMN}}]\label{complmax}{\quad }
\begin{itemize}
\item[a)] If $\ty$ is $f$-complete, then it is $f$-maximal.
\item[b)]  If $\ty$ is $f$-maximal, then either $\ty$ is  $f$-complete, or the output of the Main loop applied to $\ty$ is a unique branch of order $r+1$, which is $f$-complete.
\end{itemize}
\end{proposition}

Thus, the fact that all types of $\ty_r(f)$ are complete is essentially equivalent to the fact that they are all maximal.
The proof that  this will occur after a finite number of iterations is provided by the Theorem of the index.

\begin{theorem}[Theorem of the index {\cite[Thm.4.18]{GMN}}]
For all $r\ge 1$,
\begin{equation}\label{thindex}
\ind(f)\ge \ind_1(f)+\cdots+\ind_r(f),
\end{equation}
and equality holds if and only if all types of $\ty_r(f)$ are $f$-maximal.
\end{theorem}

This theorem shows that after a finite number of iterations all types of $\ty_r(f)$ will be $f$-maximal, because the sum of the right-hand side is bounded by the absolute constant $\ind(f)$. By (b) of Proposition \ref{complmax}, either $\T_r(f)$, or $\T_{r+1}(f)$, will contain only $f$-complete types. By Theorem \ref{criteri} and Proposition \ref{end}, our family of complete types determines the complete factorization of $p\,\ZK$ into a product of prime ideals. At this final stage we have necessarily an equality in (\ref{thindex}), so that we get a computation of $\ind(f)$ as a by-product.

\begin{rem}
If at the end of step 1 of the Main loop of the Basic algorithm, we accumulate to a global variable the value $\ind_{\ty}(f)$, the final output of this global variable is $\ind(f)$. In particular, $\ind_\ty(f)$ is an absolute measure of the distance covered by each iteration of the Main loop, towards the end of the algorithm.
\end{rem}

Summing up, we have proved the main theorem of the paper.
\begin{theorem}\label{factorization}
 Given a number field $K$, a generating equation $f(x)\in\Z[x]$, and a prime number $p$,
we can construct a set $\T$ of $f$-complete types, that faithfully represents $f(x)$. The types of $\T$ are in 1-1 correspondence with the prime ideals of $K$ lying above $p$, and the ramification index and residual degree of each ideal can be read from data  of the corresponding type. Along the construction of $\T$, the algorithm computes the $p$-valuation of the index of $f(x)$ as well.
\end{theorem}

The Theorem of the index and Proposition \ref{complmax} show that the number of iterations of the Main loop is bounded by $\ind(f)$. Actually, in practice, the number of iterations is much lower,
because  in each step, $\ind_\ty(f)$ is usually much bigger than one, due to the abundance of the number of points of integer coordinates below an average Newton polygon with a fixed length $\om_{r}^\ty(f)$, and the fact that these points are counted with weight $f_0\cdots f_{r-1}$.

In the next section we introduce a crucial optimization. A \emph{refinement process} will control at each iteration wether it is strictly necessary to build a type of higher order, or it is possible to keep working in the same order, to  avoid an increase of the recursivity in the computations. For instance, the polynomial $f(x)=(x-2)^2+2^{2k}$ would require the construction of types of level $\approx k$ in a strict application of the Basic algorithm, while it can be completely analyzed with a refined type of order 1.

\section{Optimal representatives of types}\label{secrefinement}
\subsection{Detection of optimal representatives}
The construction of types dividing a given polynomial is not canonical: in the construction of the representatives $\phi_r(x)$ one makes some choices, mainly related to lifting certain polynomials over finite fields to polynomials over $\Z$. A natural question arises: are there some choices better than other ones?

Consider the following trivial example: let $p=2$, $f(x)=x^2-4x+12$, and $K=\Q(\theta)=\Q(\sqrt{-2})$, with $\ZK=\Z[\sqrt{-2}]$. The polynomial $f(y)$ has only one irreducible factor, $\psi_0(y)=y$, modulo $2$; thus, the type of order zero $\ty=\psi_0(y)$ gives no information about the factorization of $2\,\ZK$. The more natural lifting of $\psi_0$ to $\Z[x]$ is $\phi_1(x)=x$, and the corresponding Newton polygon and residual polynomial determine a unique extension of $\ty$ to a type of order one, $(x;-1,y+1)$, which is still not complete, so that we must construct a type of (at least) order 2 to determine the factorization of $2\,\ZK$. If we choose instead, $\phi_1(x)=x-2$, we find $f(x)=(x-2)^2+2^3$, and the unique extension of $\ty$ to a type of  order one, $(\phi_1(x);-3/2,y+1)$, is complete. Thus, it is clear that this second choice of $\phi_1(x)$ is better.

While it seems very difficult to predict a priori whether a choice of $\phi_r(x)$ is better than another,
 it is possible a posteriori to know if our choice was optimal and, if this is not  the case, to improve its quality.

\begin{theorem}\label{optimal} Let $\ty^0\in \ty_{r-1}(f)$ be a type of order $r-1$, which is not $f$-complete, and let $\phi_r(x)$ be a representative of $\ty^0$. Let $\ty=(\tilde{\ty}^0;\la_r,\psi_r(y))\in\ty_r(f)$ be one of the branches of $\ty^0$, and let $f_{\ty}(x)\in\Z_p[x]$ be the factor of $f(x)$ determined by $\ty$.
Let $\phi_{r}'(x)\in\Z[X]$ be another representative of $\ty^0$. If $e_rf_r>1$, then,
\begin{itemize}
\item[a)] The Newton polygon $N'_r(f_\ty)$,  with respect to $\ty^0$ and $\phi'_r(x)$, is one-sided with slope $\la'_r\ge \la_r$, and it has the same end point than $N_r(f_\ty)$.
\item[b)] The residual polynomial $R_r'(f_\ty)(y)$, with respect to $\ty^0$, $\phi'_r(x)$ and $\la'_r$, has only one irreducible factor in $\F_r[y]$; that is, $R_r'(f_\ty)(y)\sim \psi'_r(y)^{a'_r}$, for some monic irreducible polynomial $\psi'_r(y)$.
\item[c)] Let $(\tilde{\ty}^0)'=(\phi_1(x);\cdots;\la_{r-1},\phi'_r(x))$ be the extension of $\ty^0$ determined by the choice of $\phi'_r(x)$, and let $\ty'=((\tilde{\ty}^0)';\la'_r,\psi'_r(y))$. If $\la'_r> \la_r$, then $e'_r=f'_r=1$. If
$\la'_r=\la_r$, then $e'_r=e_r$, $f'_r=f_r$, and $\om_{r+1}^{\ty'}(f)=\om_{r+1}^{\ty}(f)$.
\end{itemize}
\end{theorem}

Therefore, if $e_{r}f_{r}>1$, the representative $\phi_r(x)$ is optimal for this branch $\ty$ of order $r$.
The absolute measures $\ind_{\ty^0,\phi'_r}(f)$, $\ind_{\ty^0,\phi_r}(f)$ are not the right invariant to compare because they incorporate the influence of other branches. If we center our attention on the branch $\ty$, there are two situations in which the choice of $\phi'_r(x)$ would lead us to be closer to the end of the analysis of this branch:
\begin{enumerate}
 \item $\ty$ is replaced by several branches of order $r$,
\item $\ty$ is replaced by $\ty'$, with $\om^{\ty'}_{r+1}(f)<\om^{\ty}_{r+1}(f)$.
\end{enumerate}
Items a), b) show that any choice of $\phi'_r(x)$ leads to replacing the branch $\ty$ of $\ty^0$, by a single branch $\ty'$. Also, if $\la'_r=\la_r$, we have $\om_{r+1}^{\ty'}(f)=\om_{r+1}^{\ty}(f)$; thus, replacing the type $\ty$ by $\ty'$ makes no difference at all in this case. 

However, if $|\la'_r|<|\la_r|$, we get a definitely worse approximation to the final solution, because
 $\om_{r+1}^{\ty}(f)=\om_r^{\ty'}(f)/(e_rf_r)$, by (\ref{distance}). Thus, the type $\ty$ is much nearer to be complete than $\ty'$. Also, if $f_r>1$, $\ind_\ty(f)$ will be probably bigger than $\ind_{\ty'}(f)$, because each point of integer coordinates below $N_{r+1}^-(f)$ will contribute with a higher weight, $f_0\cdots f_r$, to the $f$-index.

Note that a choice of the representative $\phi_r(x)$ of $\ty^0$ can be optimal for some branches $\ty$ and non-optimal for other branches. We shall see later that the condition $e_rf_r>1$ is also necessary for the optimality of $\phi_r(x)$ with respect to the branch $\ty$.

For the proof of Theorem \ref{optimal}, we need an auxiliary result. Fix a type $\ty^0$ of order $r-1$ and dividing $f(x)$. For any $\n=(n_0,\dots,n_{r-1})\in\N^r$, denote $\Phi(\n)(x)=p^{n_0}\phi_1(x)^{n_1}\dots\phi_{r-1}(x)^{n_{r-1}}$. Let $\t\in\overline{\Q}_p$ be a root of $f_{\ty^0}(x)$, and $L=\Q_p(\t)$. In \cite[(27)]{GMN}, an embedding $\F_r\hookrightarrow\F_L$, is defined by
\begin{equation}\label{embedding}
\iota_\t\colon \F_r\hookrightarrow \F_L,\quad z_0\mapsto \bar{\t},\ z_1\mapsto \gb1,\dots,\ z_{r-1}\mapsto \gb{r-1},
\end{equation}
for certain rational functions $\gamma_i(x)\in\Z(x)$ such that $v(\gamma_i(\t))=0$ \cite[Def.2.13,Cor.3.2]{GMN}.

\begin{lemma}\label{M}Let $\ty^0,\,\t,\,L$ be as above.
Let $M(x)\in\Z[x]$ be a polynomial of degree less than $m_r$. Suppose that $\n\in\N^r$ satisfies
$v(M(\t))=v(\Phi(\n)(\t))$. Then, the nonzero element $\overline{M(\t)/\Phi(\n)(\t)}\in\F_L^*$ belongs to $\iota_{\t}(\F_r)$, and the element $\iota_{\t}^{-1}(\overline{M(\t)/\Phi(\n)(\t)})\in\F_r^*$ is independent of the choice of $\t$.
\end{lemma}

\begin{proof}
Let $J:=\{\j=(j_0,\dots,j_{r-1})\in\N^r\tq 0\le j_i<e_if_i, \mbox{ for }0\le i<r\}$, where we take $e_0=1$ by convention. Since $\deg M<m_r$, this polynomial can be written in a unique way as
$$M(x)=\sum_{\j=(j_0,\dots,j_{r-1})\in J}a_\j x^{j_0}\Phi(0,j_1,\dots,j_{r-1})(x),$$
for certain integers $a_\j$. By \cite[Lem.4.21]{GMN}, we have
$$v(a_\j)\ge \delta_\j:=v(M(\t))-v(\Phi(0,j_1,\dots,j_{r-1})(\t)),$$ for all $\j\in J$. Let
$J_0=\{\j\in J\tq v(a_\j)=\delta_\j\}$. Denote $b_\j=a_\j/ p^{\delta_\j}$, and $\j'=(\delta_\j,j_1,\dots,j_{r-1})$. We can write $M(x)$ as
$$
M(x)=\sum_{\j\in J_0}b_\j x^{j_0}\Phi(\j')(x)+N(x),
$$
where $N(x)\in\Z[x]$ satisfies $v(N(\t))>v(M(\t))$. Now,
$$
\dfrac{M(x)}{\Phi(\n)(x)}=\sum_{\j\in J_0}b_\j x^{j_0}\Phi(\j'-\n)(x)+\dfrac{N(x)}{\Phi(\n)(x)}.
$$
By hypothesis, $v(\Phi(\j'-\n)(\t))=0$. Since $\om_{r+1}(\Phi(\j'-\n))=0$ \cite[Prop.2.15]{GMN}, we have
$v_r(\Phi(\j'-\n)(x))=0$, by \cite[Prop.2.9]{GMN}. By \cite[Lem.2.16]{GMN}, there exists a sequence $i_1,\dots,i_{r-1}$ of integers, that depend only on $\j'$ and $\n$, such that
$$
\Phi(\j'-\n)(x)=\gamma_1(x)^{i_1}\cdots \gamma_{r-1}(x)^{i_{r-1}}.
$$
Hence, the element of $\F_L^*$,
$$
\overline{M(\t)/\Phi(\n)(\t)}=\sum_{\j\in J_0}\bar{b}_\j \bar{\t}^{j_0} \overline{\Phi(\j'-\n)(\t)},
$$belongs to $\iota_\t(\F_r)$. Since all the ingredients $a_\j$, $\delta_\j$, $i_1,\dots,i_{r-1}$ etc. depend only on $\ty^0$, the element $\iota_{\t}^{-1}(\overline{M(\t)/\Pi(\t)})\in\F_r^*$ is independent of $\t$.
\end{proof}

\begin{proof}[Proof of Theorem \ref{optimal}]
Let $\t\in\overline{\Q}_p$ be now a root of $f_{\ty}(x)$, and $L=\Q_p(\t)$.
Let us prove first that $v(\phi_r(\t))\ge v(\phi'_r(\t))$. In fact, let us show that $v(\phi_r(\t))<v(\phi'_r(\t))$ implies $e_r=f_r=1$. Let us write $\phi'_r(x)=\phi_r(x)+M(x)$, for certain polynomial $M(x)\in\Z[x]$, of degree less than $m_r$. If $v(\phi_r(\t))<v(\phi'_r(\t))$, then $v(M(\t))=v(\phi_r(\t))=(v_r(\phi_r)+|\la_r|)/e_1\cdots e_{r-1}$, by the Theorem of the polygon \cite[Thm.3.1]{GMN}. Since $\deg M<m_r$ we have $\om_{r+1}(M)=0$ \cite[Lem.2.2]{GMN}, and \cite[Prop.2.9]{GMN} shows that $v_r(M)=e_1\cdots e_{r-1}v(M(\t))=v_r(\phi_r)+|\la_r|$; hence $\la_r$ is an integer, and $e_r=1$.

We use now some other rational functions introduced in
\cite[Def.2.13]{GMN}, and the identity $v_r(\phi_r)=e_{r-1}f_{r-1}v_r(\phi_{r-1})$ \cite[Thm.2.11]{GMN}:
$$
\gamma_r(x)=\dfrac{\Phi_r(x)}{\pi_r(x)^{h_r}}=\dfrac{\phi_r(x)}{\pi_r(x)^{h_r}\pi_{r-1}(x)^{v_r(\phi_r)/e_{r-1}}}.
$$
Denote $\Pi(x)=\pi_r(x)^{h_r}\pi_{r-1}(x)^{v_r(\phi_r)/e_{r-1}}$. Since $v(\gamma_r(\t))=0$, we have $v(\Pi(\t))=v(\phi_r(\t))=v(M(\t))$. By \cite[(17)]{GMN}, we can write $\Pi(x)=\Phi(\n)(x)$, for some $\n\in\N^r$ that depends only on $\ty^0$. Now, if we reduce modulo $\mathfrak{m}_L$ the identity
$$
\dfrac{\phi'_r(\t)}{\Pi(\t))}=\gamma_r(\t)+\dfrac{M(\t)}{\Pi(\t))},
$$
Lemma \ref{M} shows that $\gb{r}=-\overline{M(\t)/\Pi(\t))}$ belongs to $\iota_\t(\F_r)$. Since $\gb{r}$ is a root of $\iota_\t(\psi_r(y))$ \cite[Prop.3.5]{GMN}, we get $f_r=1$.

We prove now a) of the theorem. If we show that $v(\phi'_r(\t))$ takes the same value for all the roots $\t$ of $f_\ty(x)$, then, by the Theorem of the polygon, $N'_r(f_\ty)$ will be one-sided with slope
\begin{align*}
\la'_r=v_r(\phi'_r)-e_1\cdots e_{r-1} v(\phi'_r(\t))&=v_r(\phi_r)-e_1\cdots e_{r-1} v(\phi'_r(\t))\\&\ge
v_r(\phi_r)-e_1\cdots e_{r-1} v(\phi_r(\t))=\la_r.
\end{align*}
Now, if $v(\phi'_r(\t))=v(\phi_r(\t))$ for all $\t$, then the value $v(\phi'_r(\t))$ is constant, because the value $v(\phi_r(\t))$ is constant. Note that $v(M(\t))=v_r(M)/e_1\cdots e_{r-1}$ is independent of $\t$.
If there is one $\t_0$ with $v(\phi'_r(\t_0))<v(\phi_r(\t_0))$, then $v(M(\t_0))=v(\phi'_r(\t_0))<v(\phi_r(\t_0))$. Hence, $v(M(\t))<v(\phi_r(\t))$ for all $\t$, because both expressions are independent of $\t$. Thus, $v(\phi'_r(\t))=v(M(\t))$ is constant too. Finally, the polygons $N_r(f_\ty)$, $N'_r(f_\ty)$, have both end point $(\om_r^{\ty^0}(f),v_r^{\ty^0}(f))$.

We prove items b), c) simultaneously. Suppose first that $\la'_r>\la_r$; then, the Theorem of the polygon shows that $v(\phi'_r(\t))<v(\phi_r(\t))$. Arguing as above, this implies $e'_r=1$ and $\overline{\gamma'_r(\t)}\in\iota_\t(\F_r)$, with $\eta:=\iota_\t^{-1}(\overline{\gamma'_r(\t)})\in\F_r^*$ independent of $\t$. By the Theorem of the residual polynomial \cite[Thm.3.7]{GMN}, if $\t$ runs on all the roots of $f_\ty(x)$, then $\overline{\gamma'_r(\t)}$ runs on all the roots of the irreducible factors of $R'_r(f_\ty)(y)$. Hence, $R'_r(f_\ty)(y)\sim(y-\eta)^{a'_r}$, and $f'_r=1$.

Suppose now $\la'_r=\la_r$, so that $v(\phi_r(\t))=v(\phi'_r(\t))$, by the Theorem of the polygon. We distinguish two cases. If $v(M(\t))=v(\phi_r(\t))$, arguing as above we get $e_r=1$
and $\gb{r}=\overline{\gamma'_r(\t)}+\iota_\t(\eta)$, for some $\eta\in\F_r^*$ which is independent of $\t$. In this case, $R'_r(f_\ty)(y)\sim R_r(f_\ty)(y+\eta)$ is a nonzero constant times the power of an irreducible polynomial of degree $f'_r=f_r$. If $v(M(\t))>v(\phi_r(\t))$, then $\phi_r(\t)^{e_r}=\phi'_r(\t)^{e_r}+N(x)$, where $v(N(\t))>v(\phi_r(\t)^{e_r})$. Arguing as above, we get $\gb{r}=\overline{\gamma'_r(\t)}$, and this implies $R'_r(f_\ty)(y)\sim R_r(f_\ty)(y)$ and $f'_r=f_r$. This implies
$\om_{r+1}^{\ty'}(f)=\om_{r+1}^{\ty}(f)$ too, because $e_rf_r\om_{r+1}^{\ty}(f)=e'_rf'_r\om_{r+1}^{\ty'}(f)$ by (\ref{distance}).
\end{proof}

\subsection{The process of refinement}
What can be said when $e_{r}=f_{r}=1$?  In this case, we enlarge the type $\ty^0$ to an order $r$ type $\ty=(\tilde{\ty}^0;-h_r,y-\eta)$, and we find a representative $\phi_{r+1}(x)$ of $\ty$, of degree $m_{r+1}=m_r$. Let us emphasize a crucial observation.

\begin{rem}
The polynomial $\phi'_r(x):=\phi_{r+1}(x)$ can be taken too as a representative of $\ty^0$.
\end{rem}
In fact, $\phi'_r(x)$ has type $\ty^0$, and $\om_r(\phi'_r)=\deg \phi'_r/m_r=1$. We shall show that $\phi'_r(x)$ is always a better representative of $\ty^0$ than $\phi_{r}(x)$; thus, in this case $\phi_r(x)$ is never optimal.

The comparison betwen these two representatives is done by means of the following affine transformation:
$$
\HH: \R^2 \longrightarrow \R^2, \qquad \HH(x,y)=(x,y-h_r x).
$$
Note that the vertical lines of the plane are invariant under this transformation, and $\HH$ acts as a translation on them. Also, $\HH$ preserves points of integer coordinates. If $S$ is a side of negative slope, of length $\ell$, slope $\la$ and degree $d$, then $\HH(S)$ is a side of length $\ell$, slope $\la-h_r$ and degree $d$.

\begin{definition}
Let $h$ be a positive integer, $\ty$ a type of order $r-1$, $\phi_r(x)$ a representative of $\ty$, and $P(x)\in\Z[x]$ a nonzero polynomial that is not divisible by $\phi_r(x)$.

If $N$ is a principal polygon, we denote by $N^h$ the part of $N$ formed by the sides of slope less than $-h$.
We define
\begin{equation}\label{cuttingind}\ind_\ty^h(P):=\ind_{\ty,\phi_r}^h(P):=f_0\cdots f_{r-1}\left(\ind(N_r^h(P))-\frac 12h\ell(\ell-1)\right),
\end{equation}
where $\ell=\ell(N_r^h(P))$, and the Newton polygon is taken with respect to $\ty$ and $\phi_r(x)$.
\end{definition}
This number $\ind_\ty^h(P)$ is equal to $f_0\cdots f_{r-1}$ times the number of points of integer coordinates in the region of the plane determined by the points that lie below (or on) $N_r^{h}(P)$, strictly above the line $L_{-h}$ of slope $-h$ that passes through the last point of the polygon, and strictly beyond the vertical axis. The term $h\ell(\ell-1)/2$ takes care of the points of integer coordinates in the triangle determined by  $L_{-h}$,  the vertical axis and the horizontal line that passes through the last point of $N_r^{h}(P)$.

Let us introduce some notation. Let $N_{r+1}(-)$, denote the Newton polygon with respect to $\ty$ and $\phi_{r+1}(x)$. Let
$N'_r(-)$ denote the Newton polygon with respect to $\ty^0$ and $\phi'_r(x)=\phi_{r+1}(x)$. For any negative rational number $\la$, let $R'_{\la}(-)(y)\in\F_r[y]$ denote the residual polynomial in order $r$, with respect to $\ty^0$, $\phi'_r(x)$, $\la$, and let $R_{\la}(-)(y)\in\F_r[y]$ denote the residual polynomial in order $r+1$, with respect to $\ty$, $\phi_{r+1}(x)$, $\la$.

\begin{proposition}\label{refinement}
Let $\ty^0\in\ty_{r-1}(f)$ be a non-complete type of order $r-1$, and let $\ty=(\tilde{\ty}^0;-h_r,y-\eta)\in\ty_r(f)$ be a branch of $\ty^0$ such that $e_r=f_r=1$. Let $\phi_{r+1}(x)$ be a representative of $\ty$, and let $\phi'_r(x)=\phi_{r+1}(x)$ be the same polynomial, considered as a representative of $\ty^0$. Let $P(x)\in\Z_p[x]$ be a nonzero polynomial. Then, with the previous notations,
\begin{itemize}
\item[a)]  $(N')_r^{h_r}(P)=\HH(N_{r+1}^-(P))$,
\item[b)] $\ind_\ty(P)=\ind_{\ty^0,\phi'_r}^h(P)$,
\item[c)] There exists a nonzero constant $\epsilon\in\F_r$ that depends only on $\ty^0$, such that, for any $\la=-h/e$, with $h,e$ positive coprime integers, we have $R_{\la}(P)(y)=\epsilon^sR'_{\la-h_r}(P)(\epsilon^ey)$, where $s$ is the initial abscissa of the $\la$-component of $N_{r+1}^-(P)$ \cite[\S1.1]{GMN}.
\end{itemize}
\end{proposition}

\begin{proof}
We shall denote by $v_{r+1}$ the $p$-adic valuation attached to $\ty$, and by $v_r$ the  $p$-adic valuation attached to $\ty^0$.
Consider the $\phi_{r+1}$-adic deve\-lopment of $P(x)$, which is simultaneously its $\phi'_r$-adic development:
$$
P(x)=\sum_{0\le i}a_i(x)\phi_{r+1}(x)^i=\sum_{0\le i}a_i(x)\phi'_r(x)^i.
$$
For any $0\le i$, denote $u_i=v_{r+1}(a_i\phi_{r+1}^i)$, $u_i'=v_r(a_i(\phi'_r)^i)$, so that the points $(i,u_i)$ determine the Newton polygon $N_{r+1}(P)$, and the points  $(i,u'_i)$ determine the Newton polygon $N'_r(P)$.
Since $\deg a_i<m_r=m_{r+1}$,
$$
v_r(a_i)=e_1\cdots e_{r-1} v(a_i(\t))=e_1\cdots e_{r-1}e_r v(a_i(\t))=v_{r+1}(a_i),
$$
where $\t$ is any root of $f_\ty(x)$ in $\overline{\Q}_p$. By \cite[Prop.2.7+Thm.2.11]{GMN},
$$
v_{r+1}(\phi_{r+1})=f_re_rv_{r+1}(\phi_r)=v_{r+1}(\phi_r)=e_rv_r(\phi_r)+h_r=v_r(\phi_r)+h_r=v_r(\phi'_r)+h_r.
$$
This proves a), because $u_i=v_{r+1}(a_i\phi_{r+1}^i)=v_r(a_i(\phi'_r)^i)+ih_r=u'_i+ih_r$. 

Item b) is an immediate consequence, because $\HH$ transforms the horizontal line that passes through the last point of $N_{r+1}^-(P)$ into the line $L_{-h_r}$ of slope $-h_r$ that passes through the last point of $N_r^{h}(P)$ (cf. the figure below).

Let us prove c). The definition of the residual coefficients and the residual polynomials is given in \cite[Defs.2.20-2.21]{GMN}. Denote $N'=(N')_r^{h_r}(P)$. To every integer abscissa, $0\le i\le \ell(N')$, one attaches a residual coefficient $c_i$ of $N_{r+1}^-(P)$,
and a residual coefficient $c'_i$ of $N'$, given by
$$c_i=\left\{\begin{array}{ll}
 z_r^{t_r(i)}R_r(a_i)(z_r),&\mbox{ if $(i,u_i)$ lies on }N_{r+1}^-(P),\\
0,&\mbox{ otherwise}.
\end{array}
\right.
$$
$$c'_i=\left\{\begin{array}{ll}
 z_{r-1}^{t'_{r-1}(i)}R_{r-1}(a_i)(z_{r-1}),&\mbox{ if $(i,u'_i)$ lies on }N',\\
0,&\mbox{ otherwise}.
\end{array}
\right.
$$
By a), the points $(i,u_i)$, $(i,u'_i)$, lying on the respective polygons have the same abscissas. Suppose that $i$ is such an abscissa. For $j=r,r-1$, denote by $s_j(a_i)$ the initial abscissa of the $\la_j$-component of $N_j(a_i)$ \cite[\S1.1]{GMN}. Since $e_r=1$, we can choose $\ell_r=0$. Since $\deg(a_i)<m_r$, the polygon $N_r(a_i)$ is reduced to the point $(0,v_r(a_i))$. This implies that $t_r(i)=(s_r(a_i)-\ell_ru_i)/e_r=0$; also, $R_r(a_i)(y)$ is a constant, equal to $z_{r-1}^{t_{r-1}(0)}R_{r-1}(a_i)(z_{r-1})$. The exponents $t'_{r-1}(i)$, and $t_{r-1}(0)$ are given by
$$
t'_{r-1}(i)=(s_{r-1}(a_i)-\ell_{r-1}v_r(a_i(\phi'_r)^i))/e_{r-1}, \
t_{r-1}(0)=(s_{r-1}(a_i)-\ell_{r-1}v_r(a_i))/e_{r-1}.
$$
Hence, $c_i=\epsilon^ic'_i$, where $\epsilon=(z_{r-1})^{\ell_{r-1}v_r(\phi_r)/e_{r-1}}$. Since $R_{\la}(P)(y)=c_s+c_{s+e}y+\cdots+c_{s+de}y^d$, $R'_{\la-h_r}(P)(y)=c'_s+c'_{s+e}y+\cdots+c'_{s+de}y^d$, we get
$R_{\la}(P)(y)=\epsilon^sR'_{\la-h_r}(P)(\epsilon^ey)$.
\end{proof}
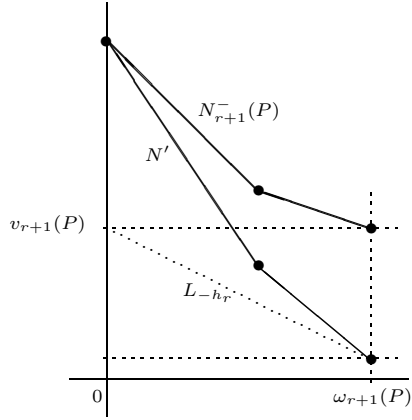
\begin{center}
\setlength{\unitlength}{5.mm}
\begin{picture}(11,11)
\put(-.2,8.8){$\bullet$}\put(3.85,4.85){$\bullet$}\put(6.85,3.85){$\bullet$}
\put(4,5){\line(-1,1){4}}\put(7,4.03){\line(-3,1){3}}
\put(4,5.03){\line(-1,1){4}}\put(7,4){\line(-3,1){3}}
\put(3.85,2.85){$\bullet$}\put(6.85,.35){$\bullet$}
\put(4,3){\line(-2,3){4}}\put(7,.5){\line(-6,5){3}}
\put(4,3.03){\line(-2,3){4}}\put(7,.502){\line(-6,5){3}}
\put(-1,0){\line(1,0){9}}\put(0,-1){\line(0,1){11}}
\multiput(-.1,4)(.25,0){32}{\hbox to 1pt{\hrulefill }}
\multiput(-.1,.55)(.25,0){32}{\hbox to 1pt{\hrulefill }}
\multiput(-.1,4)(.2,-.1){35}{\mbox{\begin{scriptsize}.\end{scriptsize}}}
\multiput(7,-.1)(0,.25){21}{\vrule height1pt}
\put(2,2.2){\begin{scriptsize}$L_{-h_r}$\end{scriptsize}}
\put(-.4,-.6){\begin{scriptsize}$0$\end{scriptsize}}
\put(-2.6,4){\begin{scriptsize}$v_{r+1}(P)$\end{scriptsize}}
\put(6,-.6){\begin{scriptsize}$\om_{r+1}(P)$\end{scriptsize}}
\put(2.4,7){\begin{scriptsize}$N_{r+1}^-(P)$\end{scriptsize}}
\put(1,5.8){\begin{scriptsize}$N'$\end{scriptsize}}
\end{picture}
\end{center}\bigskip\bigskip

Proposition \ref{refinement}, applied to $P(x)=f(x)$, shows that
 $\phi'_r(x)$ is a better representative of $\ty^0$ than $\phi_r(x)$, in what the analysis of the branch $\ty$ concerns.
Actually, we have proved something stronger: the information obtained by applying to $\ty$ the Basic algorithm in order $r+1$, is exactly the same information obtained by applying the Basic algorithm to $\ty^0$ in order $r$, as long as we take $\phi'_r(x)$ as a representative, we analyze $(N')^{h_r}_r(f)$ instead of the whole $(N')_r^-(f)$, and we replace $\ind_\ty(f)$ by $\ind^{h_r}_{\ty^0,\phi'_r}(f)$. By ``to obtain the same information" we mean to obtain the same number of new branches, the same (decreased) value of $\om_{r+1}(f)$ for each of them, and to cover the same distance to the end of the analysis of the branch $\ty$.

Moreover, let $\la\in\Q^-$ be the slope of a side of $N_{r+1}^-(f)$, $\psi(y)$ a monic irreducible factor of $R_\la(f)(y)$, and $\ty''=(\tilde{\ty}^0;-h_r,\phi_{r+1};\la,\psi(y))$ the corresponding branch of $\ty$ of order $r+1$. By Proposition \ref{refinement}, this branch mirrors a branch of order $r$, $\ty'=((\tilde{\ty}^0)';\la-h_r,c^{\deg\psi}\psi(c^{-1}y))$, of $\ty^ 0$, with respect to the choice of $\phi'_r(x)$ as a representative.

\begin{corollary}\label{feq}
The types $\ty'$ and $\ty''$ are $f$-equivalent.
\end{corollary}

\begin{proof}
If $\la=h_{r+1}/e_{r+1}$, with $h_{r+1},\,e_{r+1}$ positive coprime integers, then $\la-h_r=(h_{r+1}-e_{r+1}h_r)/e_{r+1}$, with coprime numerator and denominator; thus, $e^{\ty''}_{r+1}=e^{\ty'}_r=e_{r+1}$. Also, $f^{\ty''}_{r+1}=f^{\ty'}_r=\deg \psi$, so that  $m^{\ty''}_{r+2}=e_{r+1}\deg\psi \deg\phi_{r+1}= m^{\ty'}_{r+1}$. By c) of Proposition \ref{refinement}, applied to $P(x)=f_{\ty'}(x)$, we have $\om^{\ty''}_{r+2}(f_{\ty'})=\om^{\ty'}_{r+1}(f_{\ty'})$, and
$$
\deg f_{\ty'}=m^{\ty'}_{r+1}\om^{\ty'}_{r+1}(f_{\ty'})=m^{\ty''}_{r+2}\om^{\ty''}_{r+2}(f_{\ty'}).
$$ 
This shows that $f_{\ty'}(x)$ has type $\ty''$; since $\deg f_{\ty'}=\deg f_{\ty''}$, we have necessarily $f_{\ty'}=f_{\ty''}$.
\end{proof}

This observation leads to an important optimization of the Basic algorithm. Whenever we apply the Main loop to a type $\ty^0$ and one of the outputs is a non-complete branch $\ty=(\tilde{\ty}_0;-h_r,y-\eta)$, with $e_r=f_r=1$:

\begin{enumerate}
\item we replace $\ty$ by the type $\ty^0$ itself, but taking $\phi'_r(x)=\phi_{r+1}(x)$ as a new representative,
\item we store the cutting slope $h_r$ as one of the data of the last level of $\ty^0$, so that when the turn comes to apply the Main loop to the new $\ty^0$,  only the sides of slope less than $-h_r$
of $(N')_r^-(f)$ will be analyzed.
\end{enumerate}

We call this a \emph{refinement step}, and we use the term \emph{Montes' algorithm} to name the algorithm that is obtained from the Basic algorithm by applying a refinement step to every branch with $e_rf_r=1$.
By Corollary \ref{feq}, every future branch $\ty'$ of $\ty$ is replaced by an equiva\-lent branch $\ty''$ of the new $\ty^0$, so that the two algorithms are equivalent. By Proposition \ref{refinement}, the distance to the end of the algorithm covered by one application of the Main loop of Montes' algorithm is measured by $\ind^{h_r}_{\ty^0,\phi'_r}(f)$.

However, the refinement steps cause a strong diminution of the complexity. In fact, passing from order $r$ to order $r+1$ introduces a new level of recursivity in the basic tasks of the Main loop. Therefore, in Montes' algorithm the same information is obtained by working in a lower order.  For instance, if in the Basic algorithm we find a branch of order $r+n$, with $n$  successive levels with $e_{r+i}f_{r+i}=1$, for $0\le i<n$,
$$
\tilde{\ty}=(\phi_1(x);\la_1,\phi_2(x);\cdots;\la_{r-1},\phi_r(x);-h_r,\phi_{r+1}(x);\cdots;-h_{r+n-1},\phi_{r+n}(x)).
$$
Starting with the trunk $(\phi_1(x);\la_1,\phi_2(x);\cdots;\la_{r-1},\phi_r(x))$, we reached $\tilde{\ty}$ by applying the Main loop $n$ times in orders $r,r+1,\dots, r+n$. If we refine, this type collapses to
$$
\tilde{\ty}'=(\phi_1(x);\la_1,\phi_2(x);\cdots;\la_{r-1},\phi'_r(x)),
$$with $\phi'_r(x)=\phi_{r+n}(x)$,
and we reach $\tilde{\ty}'$ by applying the Main loop $n$ times too, but always in order $r$.

In order to homogenize the flow of the algorithm, we introduce a variable $H_r$ that stores an integer value at each $r$-th level of each type. Initially, it is given the value zero, and it is changed to $H_r=h_r$ if we fall in a refinement step. This allows us to use a general Main loop, which presents only two differences with respect to the Main loop of the Basic algorithm:

\begin{itemize}
 \item in step 1 only the sides of slope less than $-H_r$ of the Newton polygon of $f(x)$ are analyzed,
\item in step 4, after computing the representative $\phi_{r+1}$ of a non-complete new branch of order $r$, we proceed as follows: if $e_rf_r>1$, we take the new branch as one of the output branches, with the value $H_{r+1}=0$. If $e_rf_r=1$, we take the input type as one of the output branches, with $H_r=h_r$, and $\phi'_r=\phi_{r+1}$ as a new representative.
\end{itemize}

We can interpret also the refinement steps as a search for the optimal representatives. The search is performed by applying the Basic algorithm and  not enlarging the types till a branch with $e_rf_r>1$ is found.

Summing up, Montes' algorithm has the same number of iterations than the basic algorithm, but a much lower complexity. It works only with optimal representatives, and it works always at the minimum order possible till a new optimal representative forces to pass to a higher order.

In spite of this apparent strong optimality, one could speculate on an improvement based on a
more intelligent way to pass from an optimal type of order $r$ to an optimal type of order $r+1$.
The search for an optimal representative is done by blind lifting of certain polynomials over finite fields to $\Z$,
and a blind application of the Basic algorithm (without raising the order). This is extremely fast in practice, but there could exist a more direct way to obtain the next optimal representative.

\subsection{Computation of the index with Montes' algorithm}
Denote by $\ty_r^{\opt}(f)$ the set of types of order $r$ that are produced by Montes' algorithm. The sets $\ty_r^{\opt}(f)$ are quite different from the sets $\ty_r(f)$ produced by the Basic algorithm, which were crucial in the definition of $\ind_r(f)$ and the  proof of the Theorem of the index. We need to compare in some sense these two types of sets. This is provided by Proposition \ref{inverserefinement} below, which is similar to Proposition \ref{refinement}, but going in the opposite direction.

Let's go back to the situation of Corollary \ref{feq}. Suppose that the Basic algorithm is working with a type $\ty^0$ of order $r-1$, and it finds a branch of order $r+1$ (we denote it now by $\ty$ instead of $\ty''$):
$$\ty=(\tilde{\ty}^0;-h_r,\phi_{r+1};\la_{r+1},\psi_{r+1}(y))\in\ty_{r+1}(f),$$ as the result of two consecutive enlargements of $\ty_0$, and we have $e_r^{\ty}f_r^{\ty}=1$. Let $c=\epsilon^{e^\ty_{r+1}}$ be the constant of c) of Proposition \ref{refinement}, and consider the branch of order $r$ of $\ty^ 0$, $$\ty'=((\tilde{\ty}^0)';\la'_r,\psi'_r(y)),\quad \la'_r=\la_{r+1}-h_r,\ \psi'_r(y)=c^{f^\ty_{r+1}}\psi_{r+1}(c^{-1}y).$$
By Corollary \ref{feq}, $\ty'$ is $f$-equivalent to $\ty$.

\begin{rem}
Any representative $\phi'(x)$ of $\ty'$ can be taken too as a representative of $\ty$.
\end{rem}
In fact, along the proof of Corollary \ref{feq} we saw that $m^{\ty}_{r+2}=m^{\ty'}_{r+1}$, so that $\phi'(x)$ has the  right degree; thus, it is sufficient to check that $\om^{\ty}_{r+2}(\phi')=1$, and this is given by c) of Proposition \ref{refinement} applied to $P(x)=\phi'(x)$.

Let $N_{r+2}(-)$, denote the principal Newton polygon of order $r+2$, with respect to $\ty$ and $\phi(x):=\phi'(x)$. Let
$(N')_{r+1}(-)$ denote the principal Newton polygon with respect to $\ty'$ and $\phi'(x)$. Denote
$$
\F_r:=\F_r^{\ty'}=\F_r^{\ty}=\F_{r+1}^{\ty};\quad \F:=\F_{r+2}^{\ty}=\F_r[y]/\psi_{r+1}(y)=\F_r[y]/\psi_{r+1}(c^{-1}y)=\F_{r+1}^{\ty'}.
$$
For any $\la=-h/e$, with $h,e$ coprime positive integers, let $R_{\la}(-)(y)\in\F[y]$ denote the residual polynomial in order $r+2$, with respect to $\ty$, $\phi(x)$, $\la$, and let $R'_{\la}(-)(y)\in\F[y]$ denote the residual polynomial in order $r+1$, with respect to $\ty'$, $\phi'(x)$, $\la$.

Let us write $e_{r+1}=e_{r+1}^\ty=e_r^{\ty'}$, $h_{r+1}=h_{r+1}^\ty$. For the type $\ty'$ we have $\ell_r^{\ty'}h_r^{\ty'}-(\ell'_r)^{\ty'}e_r^{\ty'}=1$. Since $h_r^{\ty'}=h_{r+1}-e_{r+1}h_r$, this can be written as
$$
\ell_r^{\ty'}h_{r+1}-\left((\ell'_r)^{\ty'}_r+\ell_r^{\ty'}h_r\right)e_{r+1}=1.
$$
For the type $\ty$ we have $\ell_{r+1}^{\ty}h_{r+1}-(\ell'_{r+1})^{\ty}e_{r+1}=1$.
Therefore, we can choose $\ell_{r+1}^{\ty}=\ell_r^{\ty'}$, $(\ell'_{r+1})^{\ty}=(\ell'_{r})^{\ty'}+\ell_r^{\ty'}h_r$.

\begin{proposition}\label{inverserefinement}
Let $P(x)\in\Z[x]$ be a nonzero polynomial. Suppose we choose $\ell_{r+1}^{\ty}=\ell_r^{\ty'}$. With the previous notations,
\begin{itemize}
\item[a)]  $N_{r+2}(P)=(N')_{r+1}(P)$.
\item[b)] $\ind_{\ty}(P)=\ind_{\ty'}(P)$.
\item[c)] There exists a constant $\tau\in\F_r^*$, depending only on $\ty'$, such that for any $\la=-h/e$, with $h,e$ coprime positive integers, we have $R_{\la}(P)(y)=\tau^u R'_{\la}(P)(\tau^{-h}y)$, where $u$ is the ordinate of the initial point of the $\la$-component of $N_{r+2}(P)$.
\end{itemize}
\end{proposition}

\begin{proof}
For any polynomial $a(x)\in\Z[x]$, $$v_{r+2}^\ty(a)/e_{r+1}^{\ty}=v_{r+2}^\ty(a)/e_{r+1},\qquad v_{r+1}^{\ty'}(a)/e^{\ty'}_r=v_{r+1}^{\ty'}(a)/e_{r+1},$$ are the ordinates at the origin of the lines $L_{\la}(N_{r+1}(a))$,  $L_{\la-h_r}((N')_r(a))$,  res\-pectively \cite[Def.2.5]{GMN}. These two ordinates at the origin coincide by a) of Proposition \ref{refinement}, so that $v_{r+2}^\ty=v_{r+1}^{\ty'}$. This proves a), and b) is an immediate consequence.

Consider the $\phi$-adic development $P(x)=\sum_{0\le i}a_i(x)\phi(x)^i$, and denote $u_i=v_{r+1}^{\ty'}(a_i\phi^i)=v_{r+2}^{\ty}(a_i\phi^i)$, for all $i\ge 0$. Let $\{c_i\}_{i\ge0}$ be the residual coefficients of $N_{r+2}(P)$, and $\{c'_i\}_{i\ge0}$ be the residual coefficients of $(N')_{r+1}(P)$. Since the two polygons are constructed from the same set of points $(i,u_i)$ of the plane, we have $c_i=0$ if and only if $c'_i=0$.
Let $i$ be an integer abscissa such that the point $(i,u_i)$ lies on $N_{r+2}(P)=(N')_{r+1}(P))$, so that $c_ic'_i\ne0$. In this case, we have by definition,
$$
c_i=(z_{r+1})^{t_{r+1}(i)}R_\la(a_i)(z_{r+1}),\quad c'_i=(z'_r)^{t'_r(i)}R'_\la(a_i)(z_r).
$$
By definition, $z_{r+1}\equiv y\md{\psi_{r+1}(y)}$, and $z'_r\equiv y\md{\psi_{r+1}(c^{-1}y)}$ in $\F_r[y]$; thus, $cz_{r+1}=z'_r$. By a) of Proposition \ref{refinement} applied to $P(x)=a_i(x)$, we have $s_{r+1}^\ty(a_i)=s_r^{\ty'}(a_i)$; since $\ell_{r+1}^{\ty}=\ell_r^{\ty'}$ by hypothesis, and $e_{r+1}^{\ty}=e_r^{\ty'}$, we get $t_{r+1}(i)=t'_r(i)$. Therefore, by c) of Proposition \ref{refinement},
$$
c'_i/c_i=(z'_r/z_{r+1})^{t_r'(i)}\epsilon^{-s_r^{\ty'}(a_i)}=\epsilon^{-\ell_r^{\ty'}u_i}.
$$
If $(s,u)$ is the initial point of the $\la$-component of $N_{r+2}(P)=(N')_{r+1}(P)$, and $i=s+je$, we have $u_i=u-jh$, so that
$R_{\la}(P)(y)=\tau^u R'_{\la}(P)(\tau^{-h}y)$, for $\tau=\epsilon^{\ell_r^{\ty'}}$.
\end{proof}

Therefore, we are able to deduce from the optimal types constructed by Montes' algorithm, relevant information about the general types that would be constructed by the Basic algorithm. In  particular, by an alternative and iterative application of Propositions \ref{refinement} and
\ref{inverserefinement}, all values of $\ind_\ty(f)$, for all $\ty\in\ty_r(f)$, can be captured along the flow of Montes' algorithm.

\begin{rem}
If at the end of step 1 of the Main loop of Montes' algorithm, we accumulate to a global variable the value $\ind^{h_r}_{\ty^0}(f)$, the final output of this global variable is $\ind(f)$.
\end{rem}

In fact, if the input type $\ty^0$ is the result of an ordinary
enlargement, then $h_r=0$ and $\ind^{h_r}_{\ty^0}(f)=\ind_{\ty^0}(f)$; if $\ty^0$ is the result of a refinement step then, by Proposition \ref{refinement},
$\ind^{h_r}_{\ty^0}(f)$ is equal to the $f$-index of the type that the Basic algorithm would have produced if we had not refined.
Proposition \ref{inverserefinement} guarantees that the future development of Montes's algorithm after a refinement step, mirrors
the future development of the Basic algorithm.

\section{Generators of the prime ideals}\label{secgenerators}
In this section we compute generators of the prime ideals lying above $p$ in terms of the output of Montes' algorithm:
a list $\T=\{\ty_{\p_1},\dots,\ty_{\p_1}\}$, of $f$-complete types with optimal representatives, which are in $1-1$ correspondence
with the prime ideals $\p_1,\dots,\p_g$ of $K$ dividing $p\,\ZK$. We write $e_r^\p$, $\la_r^\p$, $\phi_r^\p$, etc. to indicate that a datum corresponds to the type $\ty_\p$. Recall that $e(\p/p)=e_1^\p\cdots e_r^\p$ and $f(\p/p)=f_0^\p\cdots f_r^\p$, can be read in the data of $\ty_\p$. We choose a root $\t\in\overline{\Q}$ of $f(x)$, and denote by $\t^{\p}\in\overline{\Q}_p$ the root of $f_{\ty_\p}(x)$, image of $\t$ under the topological embedding $K\hookrightarrow K_\p$.

If $\ty\in\T$ has order zero, and $\phi(x)$ is a representative of $\ty$, then the correspon\-ding prime ideal is generated by $(p,\phi(x))$ by Kummer's criterion. If $\ty\in\T$ has order one and its truncation $\ty_0$ of order zero has $\ind_{\ty_0}(f)=0$,
then the program computes generators of the corresponding prime ideal by using Dedekind's criterion. 

From now on, we fix a type $\ty=\ty_\p$, corresponding to a prime ideal $\p$ that did not fall in those special cases. We omit the superscript $(\ )^\p$ for the data of $\ty$. Let $r$ be the order of $\ty$. We want to compute an integral element $\alpha=\alpha_\p\in\Z_K$ satisfying
$$
v_\p(\alpha)=1;\quad  v_\q(\alpha)=0, \ \forall \q\mid p,\,\q\ne \p;\quad v_\l(\alpha)\ge0,\ \forall \l\nmid p,
$$
so that the ideal $\p$ is generated by $p$ and $\alpha$.

Let us first construct an element $\beta=\beta_\p\in K$ such that $v_\p(\beta)=1$. To this end we compute first a representative
$\phi_{r+1}(x)$ of $\ty$. Since $\ty$ is $f$-complete, $\om_{r+1}(f)=1$ and $\phi_r(x)\ne f(x)$. The Newton polygon $N_{r+1}(f)$, with respect to $\ty$ and $\phi_{r+1}(x)$, is one-sided of length one, and integer slope $-H$, where $H$ is the height of the side. By the theorem of the polygon,
\begin{align*}
v_\p(\phi_{r+1}(\t))=&\,e(\p/p)v(\phi_{r+1}(\t^{\p}))=v_{r+1}(\phi_{r+1})+H\\
=&\,e_rf_rv_{r+1}(\phi_r)+H=e_rf_r(e_rv_r(\phi_r)+h_r)+H,
\end{align*}
the last two equalities by \cite[Thm.2.11,Prop.2.7]{GMN}. On the other hand, $\om_{r+1}(\phi_r)=0$ \cite[Prop.2.15]{GMN},  and \cite[Prop.2.9]{GMN} shows that
$$
v_\p(\phi_r(\t))=e(\p/p)v(\phi_r(\t^{\p}))=e_r(v_r(\phi_r)+(h_r/e_r))=e_rv_r(\phi_r)+h_r.
$$
Therefore, if we consider the element $\beta\in K$, defined as
$$
\beta:=\dfrac{\phi_{r+1}(\t)}{\phi_r(\t)^{e_rf_r}},
$$
we have $v_\p(\beta)=H$. Thus, our aim is to find a kind of ``worse possible"\ representative $\phi_{r+1}(x)$ of $\ty$; that is,  one satisfying $H=1$. To this end, we compute a blind $\phi_{r+1}(x)$. If $H=1$ we are done; if $H>1$ we use a subroutine based on \cite[Prop.2.10]{GMN}, to construct a polynomial $P(x)\in\Z[x]$ with the following properties:
$$
\deg P<m_{r+1},\quad v_{r+1}(P)=v_{r+1}(\phi_{r+1})+1,\quad R_r(P)(y)=1.
$$
The point is that $\tilde{\phi}_{r+1}:=\phi_{r+1}(x)+P(x)$ is another representative of $\ty$, and it has $H=1$. In fact,
$\deg \tilde{\phi}_{r+1}=\deg \phi_{r+1}$ and $\om_{r+1}(\tilde{\phi}_{r+1})=\om_{r+1}(\phi_{r+1})=1$, by \cite[Prop.2.8]{GMN}, so that $\tilde{\phi}_{r+1}$ is of type $\ty$ and it
$\tilde{\phi}_{r+1}$ is a representative of $\ty$. Now, since $\deg P<m_{r+1}$, we have
$$
v(P(\t^{\p}))=\dfrac{v_{r+1}(P)}{e(\p/p)}=\dfrac{v_{r+1}(\phi_{r+1})+1}{e(\p/p)}<\dfrac{v_{r+1}(\phi_{r+1})+H}{e(\p/p)}=v(\phi_{r+1}(\t^{\p})).
$$Thus, $v(\tilde{\phi}_{r+1}(\t^{\p}))=v(P(\t^{\p}))=(v_{r+1}(\phi_{r+1})+1)/e(\p/p)$, and
$$
\beta:=\dfrac{\tilde{\phi}_{r+1}(\t)}{\phi_r(\t)^{e_rf_r}}\Longrightarrow v_\p(\beta)=1.
$$

Our next step is to compute the values $v_\q(\beta)$, for all other primes $\q$ lying above $p$, $\q\ne \p$.
\begin{definition}\label{ordre-extes}
We say that $\ty_\q$ dominates $\ty$, and we write $\ty_\q>\ty$, if $\ty_\q$ is a branch of $(\ty)_{r-1}$ originated from a side of $N_{\ty,\phi_r}(f)$ of slope $\la<\la_r$. In this case we denote $\la_{\p}^{\q}=\la$ and we call it the dominating slope of $\ty_\q$ over $\ty=\ty_\p$.
\end{definition}

\begin{proposition}\label{vqbeta}
 Let $\q$ be a prime ideal of $K$ lying above $p$, $\q\ne\p$. Let $s$ be the order of $\ty_\q$. Then,
$$
v_\q(\beta)=\left\{\begin{array}{ll}
 e_rf_r(e_r^\q\cdots e_s^\q)(\la_{\p}^{\q}-\la_r),&\quad \mbox{ if }\ty_\q>\ty,\\
0 ,&\quad \mbox{ otherwise }.
\end{array}
\right.
$$
\end{proposition}

\begin{proof}
Let $r_0$ be minimal with the property $(\ty_\q)_{r_0}\ne\ty_{r_0}$. Necessarily $r_0\le\min\{r,s\}$, because the types $\ty$, $\ty_\q$ are complete. Let us deal first with the case $r_0<r$. Since $(\ty_\q)_{r_0-1}=\ty_{r_0-1}$, for some primitive choice $\phi_{r_0}(x)$ of a representative of this type, the Main loop produced (at least) two different branches, that later developed to produce the types $\ty$, $\ty_\q$. Some of these branches might have been refined, causing a change of this representative at level $r_0-1$; however, arguing with these primitive branches if it were necessary, we can assume that $\phi_{r_0}=\phi_{r_0}^\p=\phi_{r_0}^\q$. This might change the values of the data of the $r$-th level, but this is not relevant in our arguments. After our assumption, $v_i^\p=v_i^\q$, $N_i^\p(-)=N_i^\q(-)$, for $1\le i\le r_0$. We claim that
\begin{equation}\label{viq}
v_i^\q(\phi_{r+1})=v_i^\q(\phi_r^{e_rf_r}),\quad\mbox{ for all }1\le i\le r_0+1.
\end{equation}
Let us show this by induction; clearly, $v_1^\q(\phi_{r+1})=0=v_1^\q(\phi_r^{e_rf_r})$, because both polynomials are monic. Suppose that (\ref{viq}) holds for some $1\le i\le r_0$, and let us show that it holds for $i+1$. By the definition of $v_{i+1}^\q$, we need only to show that
$$N_i^\q(\phi_{r+1})=N_i^\q(\phi_r^{e_rf_r}).
$$Since $i\le r_0$, it is sufficient to check that $N_i^\p(\phi_{r+1})=N_i^\p(\phi_r^{e_rf_r})$. Now, these polygons are both one-sided of slope $\la_i$, and have the same length because the two polynomials, $\phi_{r+1}$, $\phi_r^{e_rf_r}$, have the same degree. Finally, the two polygons have the same end point by the equality of (\ref{viq}) for our $i$, and by \cite[Lem.2.17]{GMN}. This ends the proof of (\ref{viq}).

We claim now that
\begin{equation}\label{omq}
\om^\q_{r_0+1}(\phi_{r+1})=\om^\q_{r_0+1}(\phi_r)=0.
\end{equation}
In fact, the polygon $N_{r_0}^\q(\phi_r)=N_{r_0}^\p(\phi_r)$ is one-sided of slope $\la_{r_0}$. If $\la_{r_0}^\q\ne\la_{r_0}$, then $R_{r_0}^\q(\phi_r)$
is a constant and $\om^\q_{r_0+1}(\phi_r)=0$. If $\la_{r_0}^\q=\la_{r_0}$, but $\psi_{r_0}^\q\ne\psi_{r_0}$, then
$R_{r_0}^\q(\phi_r)=R_{r_0}^\p(\phi_r)$ is a power of $\psi_{r_0}$ up to a multiplicative constant, and
$\psi_{r_0}^\q\nmid R_{r_0}^\q(\phi_r)$, so that $\om^\q_{r_0+1}(\phi_r)=0$ too. The same argument works for $\phi_{r+1}$.

Finally, (\ref{viq}) and (\ref{omq}) show that $v(\phi_{r+1}(\t^\q))=v(\phi_r(\t^\q)^{e_rf_r})$. Therefore, $v_\q(\phi_{r+1}(\t))=v_\q(\phi_r(\t)^{e_rf_r})$, and $v_\q(\beta)=0$.

Assume now $r_0=r$. The Main loop applied to $(\ty_\q)_{r-1}=\ty_{r-1}$ and the representative $\phi_r=\phi_r^\p$ produced a complete branch $\ty$, and (at least) another branch of order $r$
$$
\ty_\q^0=(\phi_1(x);\cdots;\la_{r-1},\phi_r(x);\la_r^0,\psi_r^0(y)),
$$whose further development produced the type $\ty_\q$. Let us denote by $v_{r+1}^0$, $\om_{r+1}^0$, $e_r^0$ the data attached to this type. Since $\t^\q$ is the root of a $p$-adic polynomial of type $\ty_\q^0$, the Theorem of the polygon shows that,
\begin{equation}\label{viq2}
v(\phi_r(\t^\q))=(v_r(\phi_r)+|\la_r^0|)/e_1\cdots e_{r-1}.
\end{equation}
Let us compute now $v(\phi_{r+1}(\t^\q))$. The Newton polygon $N_r^0(\phi_{r+1})=N_r^\p(\phi_{r+1})$ has length $e_rf_r$, slope $\la_r$, and the ordinate of the last point is $v_r(\phi_{r+1})=v_r(\phi_r^{e_rf_r})$. Consider a line of slope $\la_r^0$ far beyond the polygon, and move it upwards till it touches it; let $L$ be this line of slope $\la_r^0$ that first touches the polygon, and let $H$ be the ordinate of $L$ at the origin.
We distinguish two cases, according to
$\la_r^0<\la_r$ or $\la_r^0\ge\la_r$. In the first case, $\la_r^0=\la_\p^\q$ is the dominating slope of $\q$ over $\p$.

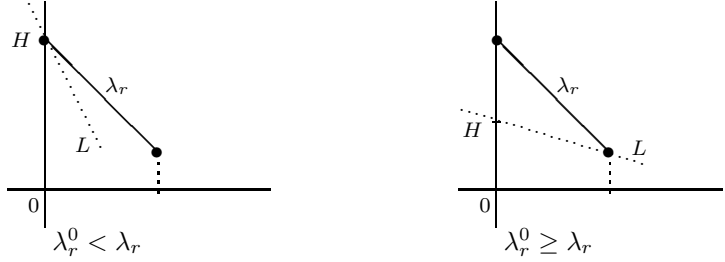
\begin{center}
\setlength{\unitlength}{5.mm}
\begin{picture}(15,5.5)
\put(-.2,3.8){$\bullet$}\put(2.8,.8){$\bullet$}
\put(-1,0){\line(1,0){7}}\put(0,-1){\line(0,1){6}}
\put(3,1){\line(-1,1){3}}\put(3.02,1){\line(-1,1){3}}
\multiput(-.5,4.9)(.1,-.2){20}{\mbox{\begin{scriptsize}.\end{scriptsize}}}
\multiput(3,-.1)(0,.25){5}{\vrule height1pt}
\put(-.45,-.6){\begin{footnotesize}$0$\end{footnotesize}}
\put(1.6,2.6){\begin{footnotesize}$\la_r$\end{footnotesize}}
\put(.8,1){\begin{footnotesize}$L$\end{footnotesize}}
\put(-.9,3.8){\begin{footnotesize}$H$\end{footnotesize}}
\put(.2,-1.6){$\la_r^0<\la_r$}
\put(14.8,.8){$\bullet$}\put(11.85,3.8){$\bullet$}
\put(11,0){\line(1,0){7}}\put(12,-1){\line(0,1){6}}
\put(15,1){\line(-1,1){3}}\put(15.02,1){\line(-1,1){3}}
\multiput(11,2.1)(.2,-.06){25}{\mbox{\begin{scriptsize}.\end{scriptsize}}}
\multiput(15,-.1)(0,.25){5}{\vrule height1pt}
\put(11.55,-.6){\begin{footnotesize}$0$\end{footnotesize}}
\put(13.6,2.6){\begin{footnotesize}$\la_r$\end{footnotesize}}
\put(15.6,.9){\begin{footnotesize}$L$\end{footnotesize}}
\put(11.1,1.4){\begin{footnotesize}$H$\end{footnotesize}}
\put(11.9,1.8){\line(1,0){.2}}
\put(12.2,-1.6){$\la_r^0\ge\la_r$}
\end{picture}
\end{center}\bigskip\bigskip

Arguing as in the case $r_0<r$, we get $\om_{r+1}^0(\phi_{r+1})=0$ in both cases, and
\begin{equation}\label{viq3}
v(\phi_{r+1}(\t^\q))=v_{r+1}^0(\phi_{r+1})/e_1\cdots e_{r-1}e_r^0=H/e_1\cdots e_{r-1},
\end{equation}
the last equality by the definition of $v_{r+1}^0$. The figures show that $H=v_r(\phi_r^{e_rf_r})+H'$, with
$$
H'=\left\{\begin{array}{ll}
e_rf_r|\la_r|,&\quad\mbox{ if }\la_r^0<\la_r,\\
e_rf_r|\la_r^0|,&\quad\mbox{ if }\la_r^0\ge\la_r.
\end{array}
\right.
$$
Now, (\ref{viq2}) and (\ref{viq3}), show that
$$
\dfrac{v_\q(\beta)}{e_1^\q\cdots e_s^\q}=\dfrac{v_r(\phi_r^{e_rf_r})+H'}{e_1\cdots e_{r-1}}-\dfrac{v_r(\phi_r^{e_rf_r})+e_rf_r|\la_r^0|}{e_1\cdots e_{r-1}}.
$$
Therefore, $v_\q(\beta)=0$ if $\la_r^0\ge\la_r$, and otherwise,
$$
v_\q(\beta)=\dfrac{e_1^\q\cdots e_s^\q}{e_1\cdots e_{r-1}}e_rf_r(|\la_r|-|\la_\p^\q|).
$$
\end{proof}

For the maximal types with respect to the ordering ``$>$", we take $\tilde{\alpha}_\p:=\beta_\p$. For the rest of the types
we compute recurrently:
$$
\tilde{\alpha}_\p:=\beta_\p \prod_{\ty_\q> \ty_\p } \tilde{\alpha}_\q^{-v_\q(\beta_\p)}.
$$
These elements are not far from generating the $\p_i$, since:
 $$
 v_{\q}(\tilde{\alpha}_\p)=\left\{
 \begin{array}{ll}
1  &\mbox{ if } \q=\p,\\
 0 & \mbox{ otherwise}.
 \end{array}
 \right.
 $$
Unfortunately, they could be non-integral at primes of $\ZK$ not dividing $p\,\ZK$. This can be easily arranged; we write each $\tilde{\alpha}_\p$ in the form $\tilde{\alpha}_\p=G(\theta)/b$,
with $G(x)\in \Z[x]$ and $b\in \Z$ coprime with the content of $G(x)$; and we conveniently modify $\tilde{\alpha}_\p$ into:
$$
\alpha_\p:=G(\theta)/p^{v(b)}.
$$

\section{Computational issues}
Recall that  Montes' algorithm is the optimization of the Basic algorithm of section \ref{ideals} that results from the application of the refinement process of section \ref{secrefinement}. In this section we give a more detailed description of this algorithm and we discuss some  computational aspects. We also include, as an extension of the algorithm, the computation of generators for the prime ideals as indicated in the last section.

\subsection{Outline of the algorithm}
The primary goal of Montes' algorithm is the computation of $\ind(f)$ and the construction of a set $\T$ of $f$-complete types, that faithfully represents $f(x)$.

By the recursive nature of its construction, many of the types generated by the algorithm will share many of its levels, so that  most of the computations necessary to enlarge them will be the same. Hence it is very  convenient to organize their computation in such a way that we can take profit of as much previous computations as possible. The simplest way to organize the computation of types is to store all the types being built by the algorithm in a list, which we  call {\tt STACK}. Once a type is completed, it is moved from the list {\tt STACK} to a second list called {\tt COMPLETETYPES}, which when the algorithm ends contains the final list of complete types.

The variable {\tt STACK}, as its name suggests, is a LIFO stack, which in practice determines the flow of the algorithm: the main loop of the algorithm extracts the last type from the {\tt STACK} and works it out to decide whether it is complete (in which case it  is moved to {\tt COMPLETETYPES}) or it originates a number of enlarged types, that will be added to the top of the {\tt STACK}. The program finishes when the {\tt STACK} is empty.

We could think of
the types being built along the algorithm as the branches of a tree.  The root of the tree is a node corresponding to the input polynomial $f(x)$. Every division of the branch into new subbranches is generated in every order by multiple sides of a Newton polygon of $f(x)$ and by multiple irreducible factors of the residual polynomial of each side.
The algorithm builds this tree of types from the topmost branch to the lowest one in every order. This strategy confers a certain ordering to the list {\tt COMPLETETYPES}, which is useful later on for the computation of generators of the ideals.

The computation of the $p$-index is performed along the construction of types: every time we analyze a
Newton polygon with respect to a type $\ty$, we add the number $\ind_\ty^{h_r}(f)$, given in (\ref{cuttingind}), to the variable {\tt TOTALINDEX}, whose final output is the value of $\ind(f)$.

Once the algorithm has emptied the \texttt{STACK}, the algorithm is almost finished: it remains only to
gather the information of every complete type to list the ramification indices and residual degrees of the prime ideals dividing $p\,\ZK$.

We now give a detailed outline of Montes' algorithm, using standard pseudo-code.

\vskip 4mm
\noindent{\bf MONTES' ALGORITHM}

\noindent INPUT:
\begin{itemize}
\item[-] A monic irreducible polynomial $f(x)\in\Z[x]$.
\item[-] A prime number $p\in\Z$.
\end{itemize}

\noindent OUTPUT:
\begin{itemize}
\item[-] The $p$-valuation of the index of $f(x)$ in $\ZK$.

\item[-] A list $\{(e_1,f_1),\dots, (e_g,f_g)\}$ of pairs of integers describing the factorization of $p\,\ZK$:
$$
p\,\ZK=\p_1^{e_1}\cdots\p_g^{e_g},  \qquad f(\p_i/p)=f_i.
$$
\item[-] A  list of integral elements $\alpha_1,\dots,\alpha_g\in \ZK$ such that $\p_i=p\,\ZK+\alpha_i\ZK$
\end{itemize}

\vskip 2mm

\noindent
{\bf INITIALIZATION STEPS}

\st1  Factor $f(y)=\psi_1(y)^{a_1}\cdots\psi_s(y)^{a_s}$ modulo $p$, with $\psi_i(y)\in\Fp[y]$ pairwise different monic irreducible polynomials.

\st2 Take monic polynomials $\phi_1(x),\dots, \phi_g(x)\in\Z[x]$ such that $\phi_i(y)\md{p}=\psi_i(y)$. Compute the polynomial $M(x)=(f(x)-\phi_1(x)^{a_1}\cdots\phi_s(x)^{a_s})/p$.

\st3 Initialize empty lists {\tt STACK} and {\tt COMPLETETYPES}, and set  {\tt TOTALINDEX}$\leftarrow0$.

\st4
FOR every polynomial $\phi_i(x)$  do

(Dedekind's criterion) If $a_i=1$ or $\psi_i(y)\nmid M(y)\md{p}$, output the ideal
$\p=(p,\,\phi_i(\theta))$,

with ramification index $e(\p/p)=a_i$ and residual degree $f(\p/p)=\deg \phi_i.$
Otherwise, add

to {\tt STACK} the extension $\tilde{\ty}=(\phi_i(x))$ of the type of order zero determined by $\psi_i(y)$, and set

$H_1=0$, as data of level one.

\vskip 3mm
\noindent{\bf MAIN LOOP}

\vskip 2mm

\noindent
WHILE the  \texttt{STACK } is non-empty do:

\st5 Extract the last  type $\ty^0$ from \texttt{STACK}. Let $r-1$ be its order.

\st6 Compute the Newton polygon $N_r^{H_r}(f)$, formed by the sides of slope smaller than $-H_r$.

\st7 Compute $\ind_{\ty^0}^{H_r}(f)$ using the formula (\ref{cuttingind}) and add this number to \texttt{TOTALINDEX}.

\vskip 2mm
\noindent
\hskip 2mm
FOR every side $S$ of $N_r^{H_r}(f)$ do

\stst8  Set $\lambda_r^{\ty^0}\leftarrow$ slope of $S$.

\stst9 Compute the $r$-th order residual polynomial $R_r(f)(y)\in \F_{r}[y]$.

\vskip 2mm

\stst{10}
FOR every irreducible factor $\psi(y)$ of $R_r(f)(y)$ do

\ststst{11} Make a copy $\ty$ of the type $\ty^0$, and set $\psi_r^{\ty}(y)\leftarrow \psi(y)$.

\ststst{12} Compute a representative  $\phi_{r+1}(x)\in\Z[x]$ of  $\ty$.

\ststst{13}  If $\ord_{\psi}(R_r(f))=1$ (the type is complete),
add $(\tilde{\ty};\lambda_r^{\ty},\psi^{\ty}(y))$ to {\tt COMPLETETYPES}, and continue to the next factor of $R_r(f)(y)$.

\ststst{14} If  $\deg \psi=1$ and $\lambda_r^{\ty}\in\Z$   (the type must be refined),
set $\phi_r^{\ty}(x)\leftarrow \phi_{r+1}(x)$, $H_r^{\ty}\leftarrow\lambda_r^{\ty}$, add $\ty$ to  the top of the {\tt STACK} and continue to the next factor of $R_r(f)(y)$.

\ststst{15} (Build a higher order type)
 Add $(\tilde{\ty};\lambda_r^{\ty},\psi^{\ty}(y))$ to the top of the {\tt STACK}.

\vskip 3mm
\noindent
{\bf END OF MAIN LOOP}

\vskip 3mm
\noindent
{\bf OUTPUT}

\st{16} Print the $p$-valuation of the index of $f$ in $\ZK$, given by the value of \texttt{TOTALINDEX}.

\st{17} For every type $\ty $ in the list {\tt COMPLETETYPES}, output the ramification index and residual degree of the corresponding ideal $\p$, given by:
$$
e(\p/p)=e_1^{\ty}\cdots e_{r}^{\ty},\quad
f(\p/p)=m_1^{\ty} f_1^{\ty}\cdots f_{r}^{\ty},
$$
where $m_1^\ty=\deg \phi_1^\ty$, and $r$ is the order of $\ty$

\vskip 5mm
\noindent
{\bf EXTENSION: COMPUTATION OF GENERATORS}\\
\noindent
All notations are taken from section \ref{secgenerators}

\st{18} FOR every type $\ty_\p$  in {\tt COMPLETETYPES} compute the element $\beta_{\p}=\tilde\phi_{r+1}(\theta)/\phi_r(\theta)^{e_r f_r}$.

\st{19} FOR every type $\ty_\p$  in {\tt COMPLETETYPES} compute the element
$\tilde{\alpha}_{\p}:=\beta_{\p}\prod_{\ty_\q>\ty} \tilde{\alpha}_{\q}^{-v_{\q}(\beta_\p)}$, where $v_{\q}(\beta_\p)$
is given in Proposition \ref{vqbeta}.

\st{20} FOR every type $\ty_\p$  in {\tt COMPLETETYPES} compute the element $\alpha_{\p}=G(x)/p^{v(b)}$.

To compute $\tilde{\alpha}_\p$ in step $\mathbf{19}$, it is necessary to slightly modify the algorithm in order to store all dominating slopes
$\la_\p^\q$.

\subsection{Some examples} {\hfill{\quad}}\linebreak
\noindent{\bf Example 1.}
Let us consider the irreducible polynomial
\begin{align*}
f(x):=& x^{12}-588x^{10}+476x^9+130095x^8-172872x^7-12522636x^6+ 24745392x^5\\
  &+486721116x^4-1583408736x^3-641009376x^2 +  10978063488x+59914669248,
\end{align*}
whose discriminant is
$$
\disc(F)=2^{84}\cdot 3^{64}\cdot 7^{52}
  \cdot 79^4\cdot 14159^2 \cdot 644173^2\cdot 3352073^2\,.
$$
We apply the algorithm to find the decomposition of the prime   $p=2$ in the ring of integers $\Z_K$ of the number field $K=\Q(\theta)$ generated by any root of the polynomial $f(x)$.
Since
$$
f(y)\equiv (y+1)^4\,y^8\,\md2,
$$
we find two types $\ty_1,\ty_2$ of order zero, extended respectively to $\phi_1^1(x)=x+1$, $\phi_1^2(x)=x$. The Newton polygon $N_1^{\ty_1}(f)$ has two  sides, with slopes $-3/2$ and $-1/2$ respectively, which single out two prime ideals $\p_1, \p_2$, with $e(\p_1/2)=e(\p_2/2)=2$ and $f(\p_1/2)=f(\p_2/2)=1$. The type $\ty_{\p_1}$ dominates $\ty_{\p_2}$ with dominating slope $\la_{\p_2}^{\p_1}=-3/2$.

The Newton polygon $N_1^{\ty_2}(f)$ has again two sides, with slopes $-1$ and $-1/2$, and residual polynomials
$R_{1,1}(f)=(y+1)^4$, $R_{1,2}(f)=(y+1)^2$,  respectively. Hence, the type $\ty_2$ yields two types $\ty_{2,1}, \ty_{2,2}$ of order one. The first one must be refined and the second one must be enlarged to an order $2$ type. To refine $\ty_{2,1}$, we take the new polynomial $\phi_1^{\ty_{2,1}}(x)=x+2$. The corresponding Newton polygon has only one side with slope smaller than $-1$; the slope is $-3/2$ and the residual polynomial $(y+1)^2$, so that this type must be enlarged too, to an order $2$ type. After computing their respective representatives, we have now two extended types of order one, ready to be enlarged to order $2$:
$$
\begin{array}{ll}
\ty_{2,1}=(x+2;-3/2,(x+2)^2+8), &\quad H_1=-1,    \\
\ty_{2,2}=(x;-1/2,x^2+2),&\quad H_1=0.
\end{array}
$$
The Newton polygon $N_{2}^{\ty_{2,1}}(f)$ has a unique side  with slope $-4$ and residual polynomial $(y+1)^2$, so that this type can be refined. We take, for instance, $\phi_2^{\ty_{2,1}}(x)=(x+2)^2+40$:
$$
\ty_{2,1}=(x+2;-3/2,(x+2)^2+40),\quad H_1=-1,\quad H_2=-4.
$$
The new polygon $N_2^{\ty_{2,1}}(f)$ has two
sides with slopes $-9$ and $-5$, originating two new prime ideals $\p_3, \p_4$ dividing $2\Z_K$, with
$e(\p_3/2)=e(\p_4/2)=2$, $f(\p_3/2)=e(\p_4/2)=1$. The type $\ty_{\p_3}$ dominates $\ty_{\p_4}$ with dominating slope $\la_{\p_4}^{\p_3}=-9$.

The Newton polygon  $N_{2}^{\ty_{2,2}}(f)$ has a unique side  with slope $-4$ and residual polynomial $(y+1)^2$; thus, we must refine. Take $\phi_2^{\ty_{2,2}}(x)=x^2+10$, $H_2=-4$. The
next Newton polygon has again a unique side with slope $-5$ and residual polynomial $(y+1)^2$. We refine again, taking
$\phi_2^{\ty_{2,2}}(x)=x^2+10+8x$. We get:
$$
\ty_{2,2}=(x;-1/2,x^2+8x+10),\quad H_1=0, \ H_2=-5.
$$
Now, $N_2^{\ty_{2,2}}(f)$ has two sides with slopes $-8$ and $-7$, both with residual polynomial $y+1$, thus giving two prime ideals $\p_5, \p_6$, with $e(\p_5/2)=e(\p_6/2)=2$, $f(\p_5/2)=e(\p_6/2)=1$. The type $\ty_{\p_5}$ dominates $\ty_{\p_6}$ with dominating slope $\la_{\p_6}^{\p_5}=-8$.

Summing up, $2\Z_K=(\p_1\cdots\p_6)^2$. The 2-index of $f(x)$ is $\ind_2(f)=33$, and $v(\disc(K))=18$.

Generators for the ideals $\p_i$ can be determined by using the formulas of section \ref{secgenerators}.
We find:
$$
\begin{array}{ll}
\tilde{\alpha}_1=\dfrac{(\theta+1)^2+4(\theta+1)+8}{(\theta+1)^2} \,, &
\tilde{\alpha}_2=\dfrac{(\theta+1)^2+2(\theta+1)+2}{(\theta+1)^2}\,
  \tilde{\alpha}_1^4\,, \\
\tilde{\alpha}_3=\dfrac{(\theta+2)^2+64(\theta+2)+40}{(\theta+2)^2+40} \,, &
\tilde{\alpha}_4=\dfrac{(\theta+2)^2+16(\theta+2)+40}{(\theta+2)^2+40}\,
  \tilde{\alpha}_3^4\,, \\
\tilde{\alpha}_5=\dfrac{\theta^2+40\,\theta+42}{\theta^2+8\,\theta+10} \,, &
\tilde{\alpha}_6=\dfrac{\theta^2+24\,\theta+10}{\theta^2+8\,\theta+10}\,
  \tilde{\alpha}_5\,,
\end{array}
$$
$$
\begin{array}{l}
\alpha_1= (2\t^{11} + \t^{10} + 2\t^9 -\t^8 + 2)/2; \\
\alpha_2=  (\t^{11} + \t^{10} + \t^9 -3\t^8 + 4\t^5 + 4\t^4 + 4)/4; \\
\alpha_3=  (131\t^{11} -474\t^{10} + 448\t^9 + 52\t^8 + 309\t^7 -366\t^6 -216\t^5 +256\t^4 -364\t^3\\
\qquad\qquad\qquad\qquad  + 136\t^2 -496\t + 32)/512; \\
\alpha_4=  (-27\t^{11} -2\t^{10} -80\t^9 + 12\t^8 + 19\t^7 + 10\t^6 + 72\t^5 -64\t^4+76\t^3  \\
\qquad\qquad\qquad\qquad -104\t^2+ 48\t + 32)/128; \\
\alpha_5=  (-45\t^{11} - 10\t^{10} - 18\t^9 + 64\t^8 - 7\t^7 + 50\t^6 - 30\t^5 -60\t^4+  32\t + 64)/64;\\
\alpha_6=  (33\t^{11} + 8\t^{10} + 42\t^9 - 4\t^8 - 53\t^7 - 8\t^6 - 58\t^5+ 16\t^4+ 32\t^3 + 64\t^2 - 32\t)/64.
\end{array}
$$\medskip

\noindent{\bf Example 2.} Take $p=2$ and consider the irreducible polynomial
$$
f(x):=(x^3+x+5)^{50}+2^{89} (x^3+x+5)^{25}+2^{178}\,.
$$
The algorithm takes initially $\phi_1(x)=x^3+x+1$, and finds a unique side with slope -2 and residual polynomial $(y+1)^{50}$. A refinement leads to $\phi_1(x)=x^3+x+5$, and a Newton polygon with one side, with slope $-89/25$ and irreducible residual polynomial   $y^2+y+1$. Hence, in the number field $K$ defined by any root of $f(x)$, we have
 $2\Z_K=\p^{25}$, with $f(\p/2)=6$. The 2-index of the polynomial is 13011. While this computation is almost instantaneous, the
determination with Pari  of a 2-integral basis of $K$ takes about 190 seconds, and needs an amount of 244 Mb of memory.

\subsection{Some remarks on the complexity}

We have not developed a detailed analysis of the complexity of the algorithm, but the experimental results of the next section indicate that its running time is excellent. We now provide some arguments to explain this good behaviour.

We saw in section \ref{secthindex} that the number of iteration of the Main loop is bounded by $\ind(f)$ and that each iteration covers $\ind_\ty(f)$ steps from the total value of $\ind(f)$.
One is tempted to conclude that the running time of the algorithm is linear on the discriminant of $f(x)$, but
this is not so evident, since the  treatment of higher order types is much expensive than the treatment of low order types. However, this is balanced by two facts: on one hand, $\ind_\ty(f)$ is generally much bigger than one in each iteration; on the other hand, the higher the order, the smaller is $\om_{r+1}(f)$, and this invariant tells the number of coefficients of the $\phi_r$-adic development of $f(x)$ that must be computed, the length of the Newton polygon to be analyzed, and it is an upper bound for the degrees of the residual polynomials.

At least, it seems that for polynomials whose types have bounded order, the running time will be at most linear on the discriminant. In practice, the average running time is much smaller.

On the other hand, the degree of the polynomials $\phi_k^{\ty}(x)$ appearing in an $f$-complete type $\ty$  is a divisor of the product $e(\p_{\ty}/p)f(\p_{\ty}/p)$. Every time the type is enlarged, the degree of the last $\phi_k(x)$ is multiplied by the corresponding product $e_k^{\ty} f_k^{\ty}$. Hence, for a polynomial $f(x)$ to have attached a type of a very high order, its  degree must be really huge. This explains why the algorithm works well for polynomials of high degree: the maximum of the orders of the types of a polynomial grows slowly in comparison with the degree.

The low memory requirement of the algorithm is another of its strong advantages: it is only necessary to store the polynomials $\phi_k(x)$ and $\psi_k(x)$ of the types being calculated (and some Taylor expansions to gain efficiency). This makes possible the treatment of polynomials of very high degree with scarce computational resources.

The complexity of the computation of the generators is dominated by the inversion of $\phi_r(\t)$ in $K$, which is a hard task if the degree of $\phi_r(x)$ is large.

\subsection{Implementation of the algorithm}
\label{implementacio}

The first implementation of Montes' algorithm was programmed by J. Gu\`{a}rdia  in 1997, as a part of his Ph.D. It was written for Mathematica 3.0, and it included a specific package to work with finite fields, since that version of Mathematica did not carry such a package. It is still available on request to the author.
Ten years later, we started a collaboration to make a full upgrade of the algorithm, with many optimizations both theoretical and computational, including a completely new implementation in Magma.


The computation of generators for the prime ideals becomes a heavy time-consuming task if the $p$-adic factors of $f(x)$ have large degrees. For this reason, our implementation skips this calculation by defect. If the user wants to compute the generators, a Boolean variable {\tt GENERATORS} has to be given the value ``true".

As memory requirement is not a constraint for the implementation, the program stores some intermediate results (mainly $\phi$-adic expansions) to gain speed. The main data type used by the program is a specifically designed record which contain all the relevant data of a type in a given order. To avoid massive replication of the types being computed, must of the routines access them by memory address.

The program also includes a number of routines to construct types and polynomials with a prescribed set of types.
The program and its documentation, which includes all the examples presented in this paper, can be downloaded from the web page
\begin{center}
{\tt www-ma4.upc.edu/$\sim$guardia/MontesAlgorithm.html}.
\end{center}
Any comment on the program will be welcome.

\section{Some heuristics on the complexity}

We dedicate this section to illustrate the performance of (our implementation of) Montes' algorithm with several  polynomials chosen to force its capabilities at maximum in three directions: polynomials with a unique associate type of large order, polynomials which  require a lot of refinements, and polynomials with many different types. We have also included some test polynomials found in the literature.

All the tests have been done in a personal computer, with an Intel Core Duo processor, running at 2.2 Mhz, with 3Gb of RAM memory. The reader willing to check these results on his/her own can obtain the Magma code to generate these polynomials from the web page of the program.
\\
\noindent
{\bf Example 1:} Take $p=2$. Consider the irreducible polynomials
$$\,\hskip -8mm
\begin{array}{l}
\qquad\phi_1= x^2+2^2x + 2^4 ;
\\
\qquad\phi_2=\phi_1^2+2^4x\phi_1+2^{12};
\\
\qquad\phi_3=\phi_2^4+2^{23}(x+2^2)\phi_2^2+2^{42}x\phi_1;
\\
\qquad\phi_4=\phi_3^2+2^{12}x\phi_2^3\phi_3+2^{72}\phi_1\phi_2^2+2^{101}x;
\\
\qquad\phi_5=\phi_4^3+2^{34}\phi_1\phi_2\phi_3\phi_4^2+2^{215}((x(\phi_1+2^{6})(\phi_2^3+2^{25}\phi_2)+2^{27}\phi_2)\phi_3+2^{64}(x\phi_1\phi_2^2+2^{33}));
\\
\qquad\phi_6=\phi_5^6+2^{883}x\phi_3\phi_5^3+2^{1736}((x+4)\phi_1+2^{8})\phi_2^2\phi_4;
\\
\qquad\phi_7=\phi_6^2+
2^{2351}((\phi_1\phi_2^3+2^{23}x(\phi_1+2^{6})\phi_2)\phi_4
+2^{102}(x\phi_1\phi_2^3+2^{25}((x+2^{2})\phi_1+2^{6}x)\phi_2)
)\phi_5^4\\
\qquad\qquad
+2^{3234}(((x\phi_1\phi_2^3+2^{25}(x+2^{2})(\phi_1+2^6)\phi_2)\phi_3+2^{70}(x\phi_2^2+2^{27}))\phi_4\\
\qquad\qquad
+2^{168}((x+2^{2})\phi_1+2^{6}x)\phi_2^2)\phi_5;\\
\qquad\phi_8=\phi_7^{6}+
2^{7515}((((((x+4)\phi_1+2^{6}x)\phi_2^{2}+2^{31}x)\phi_3+2^{39}x\phi_1\phi_2^{3}+2^{70}x\phi_2)\phi_4^{2}\\
\quad\qquad
+2^{104}(((x\phi_1+2^{8})\phi_2^{2}+2^{25}(x\phi_1+2^{6}(x+2^{4})))\phi_3+
 2^{38}((x+4)(\phi_1+2^{6})\phi_2^{3}+2^{27}\phi_1\phi_2))\phi_4\\
\quad\qquad
+2^{208}(((x\phi_1+2^{8})\phi_2^{2}+2^{25}x\phi_1+2^{31}(x+4))\phi_3+2^{41}\phi_1\phi_2^{3}+2^{64}(x\phi_1+2^{8})\phi_2)
)\phi_5^{5}\\
\quad\qquad
+2^{924}((
(((x+4)\phi_1+2^{6}x)\phi_2^{3}+2^{25}((x+4)\phi_1+2^{8})\phi_2)\phi_3\\
\quad\qquad  +2^{134}((\phi_1+2^{4}(x+4))\phi_2^{2}+2^{25}(\phi_1+2^{4}x))
)\phi_4^{2}\\
\quad\qquad
+2^{104}(((x+4)\phi_1\phi_2^{3}+2^{25}x(\phi_1+2^{6})\phi_2)\phi_3+2^{64}((x+4)(\phi_1+2^{6})\phi_2^{2}+2^{27}\phi_1))\phi_4\\
\quad\qquad
+2^{210}(((\phi_1+2^{4}x)\phi_2^{3}+2^{23}(x+4)(\phi_1+2^{6})\phi_2)\phi_3+2^{64}(\phi_1\phi_2^{2}+2^{25}\phi_1))
)\phi_5^{2})\phi_6\phi_7^{3}\\
\quad\qquad
+2^{20618}(((x\phi_1\phi_2^{3}+2^{31}(x+2^{6})\phi_2)\phi_3+2^{66}(\phi_1\phi_2^{2}+2^{25}(\phi_1+2^{6})))\phi_4^{2}\\
\quad\qquad
+2^{104}(((x+4)\phi_1\phi_2^{3}+2^{25}x\phi_1\phi_2)\phi_3+2^{70}((x+4)\phi_2^{2}+2^{21}\phi_1))\phi_4\\
\quad\qquad
+2^{208}((((x+4)\phi_1+2^{6}x)\phi_2^{3}+2^{25}((x+4)\phi_1+2^{8})\phi_2)\phi_3+
2^{70}(x\phi_2^{2}+2^{19}(x\phi_1+2^{8}))))\phi_5^{5}\\
\end{array}
$$
$$\,\hskip -8mm
\begin{array}{l}
\quad\quad
+2^{21567}(((x(\phi_1+2^{6}(x+4))\phi_2^{2}+2^{25}(x+4)\phi_1)\phi_3+2^{39}((x\phi_1+2^{8})\phi_2^{3}+
2^{25}(x+4)\phi_1\phi_2))\phi_4^{2}\\
\quad\quad
+2^{104}((((x+4)\phi_1+2^{6}x)\phi_2^{2}+2^{33})\phi_3+2^{64}x\phi_1\phi_2)\phi_4\\
\quad\quad
+2^{208}(x+4)(\phi_1+2^{6})\phi_2^{2}\phi_3+2^{249}(\phi_1+2^{4}x)\phi_2^{3})\phi_5^{2}.
\end{array}
$$

For each $j$, the corresponding polynomial $\phi_j$ has a unique associate complete type of order $j$, so that in the corresponding number field $K_j$ the ideal $2\Z_{K_j}$ is the power of a unique prime ideal $\p_j$.
The following table contains the degree and 2-index of the $\phi_k$, the ramification index $e_j$ and residual degree $f_j$ of $\p_j$ and the time ${\tt t}_1$ used by the program to compute them. The last column indicates the  time ${\tt t}_2$ to run the extended version of the algorithm, which includes the computation of generators for the ideals $\p_j$. All the times are expressed in seconds.
The computation of the generators in the last two rows was stopped after 24 hours of running time.

$$
\begin{array}{|c|r|r|r|r|c|c|}
\hline
\phi_j & \deg \phi_j & \operatorname{ind}(\phi_j) &e_j & f_j & {\tt t}_1 & {\tt t}_2 \\
\hline\hline
\phi_1 &    2 &       2 &  1 &   2 & 0.00 & 0.00\\
\hline
\phi_2 &    4 &      16 &  1 &   4 & 0.01 & 0.01\\
\hline
\phi_3 &   16 &     360 &  2 &   8 & 0.01 & 0.01\\
\hline
\phi_4 &   32 &    1544 &  2 &  16 & 0.01 & 0.016\\
\hline
\phi_5 & 96 &     14616 &  2 &  48 & 0.08& 0.6\\
\hline
\phi_6 & 576 &   537120 &  6 &  96 & 0.70 & 406 \\
\hline
\phi_7 & 1152 & 2153376 & 12 &  96 & 4.0&\\
\hline
\phi_8 & 6912 & 77673504 &36 & 192 & 787 &\\
\hline
\end{array}
$$


\vskip 3mm

\noindent
{\bf Example 2}:
Let $f^k(x)=(x^2+x+1)^2-p^{2k+1}$, with $p\equiv 1\pmod7$ a prime number.   When we apply Montes'algorithm to factor the ideal   $p\,\ZK$, we obtain two types of order zero with liftings $\phi_1(x)\in\Z[x]$ of degree one. For both of them the Newton polygon has only one side, with slope  $-1$ and end points $(2,0)$ and $(0,2)$, and the residual polynomial is the square of a linear factor. After approximately $2k$ total refinements, both types become
$f^k$-complete. The ideal $p\,\ZK$ splits as the product of two prime ideals with ramification index 2 and residual degree 1, and the $p$-index of $f^k(x)$ is $2k$.

This is almost the  illest-conditioned quartic polynomial for the algorithm, since the  index of every type is increased a unit per refinement in general, and the total $p$-index of $f^k(x)$ is ${2k}$. Thus, the program has to make about $2k$ iterations of the main loop.  Numerical experimentation shows that even in this worst case the running time of the algorithm is very low. In the following table we show the running time  of the programm for different values of $k$ and $p$. As before, ${\tt t}_1$ is the time in seconds to compute the index, residual degrees and ramification indices, and ${\tt t}_2$ is the  time to compute also the generators for the prime ideals.
 $$
\begin{array}{|r|r|r|r||r|r|r|r|}
\hline
p &  \operatorname{ind}(f^k) & {\tt t}_1 & {\tt t}_2 &p & \operatorname{ind}(f^k) & {\tt t}_1 & {\tt t}_2  \\
\hline\hline
  7 &  1000 & 0.57 & 0.62 &        43 & 10000 & 229 & 237  \\
\hline
  7 &  2000 & 1.95 &  2.1 &       103 & 10000 & 324 & 334 \\
\hline
  7 &  4000 &  8.7 &  9.2 &      1009 & 1000 & 2.1 & 2.6 \\
\hline
  7 &  8000 & 44.7 & 46.2 &      1009 & 2000 & 10.8 & 12.3\\
\hline
  7 & 16000 &  245 & 250 &      1009 & 4000 & 58 & 62\\
\hline
  7 & 20000 &  436    & 444 &    10^9+9 & 1000 & 10 & 12.7  \\
\hline
 13 &  1000 & 0.75 &0.85 &    10^9+9 & 2000 & 57 & 66.5\\
\hline
 13 &  2000 &  2.9 & 3.1 &    10^9+9 & 4000 & 313 & 341\\
\hline
 13 & 10000 & 131  & 135 & 10^{69}+9 &  100 & 2.8 &4.8\\
\hline
 19 & 10000 & 158  & 162 & 10^{69}+9 &  200 & 8 & 13.5 \\
\hline
 31 & 10000 & 198 &  205 & 10^{69}+9 &  400 & 29.4 &48 \\
\hline
 37 & 10000 & 214 &  221 & 10^{69}+9 & 1000 & 221 &308 \\
\hline
\end{array}
$$

\vskip 5mm

\noindent
{\bf Example 3:} Take $p=13$. We now consider a polynomial with several different types. Let
$$
\begin{array}{l}
\phi_1(x)=x^2+13^2x+13^4\cdot 3;\\
\phi_2(x)=\phi_1(x)^3+((13^{18}\cdot2));\\
\phi_3(x)=\phi_2(x)^{10}+13^{89}(x+13^{2})\phi_2(x)^5+13^{176}\phi_1(x);\\
\phi_4(x)=\phi_3(x)^2+13^{248}(12(x+13^{2})\phi_1(x)+13^{8})\phi_2(x)^6+13^{335}\cdot12\phi_1(x)^2\phi_2(x);\\
\displaystyle
f_j(x)=\prod_{k=0}^{j}\phi_4(x+k)+13^{5000},\qquad j=0,\dots,12;
\end{array}
$$
In the number field $K_j$ defined by $f_j(x)$, we have the factorization
$$
13 \Z_L=\p^5_1\cdots\p^5_j, \qquad f(\p_j/13)=24.
$$
Each prime ideal comes from a different order 4 type. The 13-index of $f_j(x)$ is $21576j$. The times to compute this index and the factorization of $13$ in the fields $K_j$ are shown in the table below.

$$
\begin{array}{|r|r|r|c|}
\hline
j & \deg f_j & \operatorname{ind}(f_j)  & {\tt t}_1  \\
\hline\hline
1 &    120 &       21576  & 0.08 \\
\hline
2 &    240 &      43152 &   0.3 \\
\hline
3 &   360 &      64728 & 0.9\\
\hline
4 &   480 &    86304 & 2.3 \\
\hline
5 & 600 &     107880 & 4.4\\
\hline
6 & 720 &   129456  & 8.2\\
\hline
7 & 840 & 151032 & 13.3\\
\hline
8 & 960 & 172608 & 20.3 \\
\hline
9 & 1080 & 194184 &  29.4\\
\hline
10 & 1200 & 215760  & 40.6\\
\hline
11 & 1320 & 237336 &   55.2\\
\hline
12 & 1440 & 258912 & 72 \\
\hline
13 & 1560 & 280488 &  92\\
\hline
\end{array}
$$

\vskip 2mm
\noindent
The computation of the generators for the number field defined by polynomial $f_1(x)$ took 20 seconds, for the number field defined by $f_2(x)$ lasted 6{\tiny 1/2} hours. The computation for $f_3$ exhausted Magma's virtual memory due to a coefficient explosion in the computation of an extended gcd.

In this example one cannot expect a linear behaviour of the time versus the number of types, because the addition of more factors to the product defining the $f_j(x)$ implies  a significant growing in the size of the coefficients of the polynomial, which has a certain impact in the running time of the algorithm.

\vskip 3mm
\noindent
{\bf Example 4:} We applied the algorithm to the list of 32 polynomials $f_1,\dots, f_{32}$ appearing in \cite[appendix D]{FPR}. The total running time for altogether was less than 0.2 seconds. We then applied the algorithm to the polynomials $F_i=f_i^2+p_i^{1000}$, where $p_i$ is the prime specified in {\em loc.cit.} for every polynomial.  In the table below we display the index of these polynomials and the running times of the algorithm.
As before, ${\tt t}_1$ denotes the time in seconds to determine the index, the residual degrees and the ramification indices, and ${\tt t}_2$ is the time of the extended algorithm that includes the computation of the generators of the ideals.

$$
\begin{array}{|r|r|r|r|r||r|r|r|r|r|}
\hline
f& p  &\operatorname{ind}(f)  & {\tt t}_1   & {\tt t}_2 &
f& p  &\operatorname{ind}(f)  & {\tt t}_1   & {\tt t}_2  \\ \hline\hline
F_1    & 2     & 2502150 & 2.45  & 3.1      & F_{17 } & 2  & 1571054  & 2.9     & 7     \\ \hline
F_2    & 2     & 1481141 & 2     & 2.7      & F_{18 } & 7  & 7331055  & 77      & 93.9    \\ \hline
F_3    & 3     & 2570992 & 6.6   & 8.2      & F_{19 } & 71 & 187219   & 116.1   & 249.3   \\ \hline
F_4    & 3     & 1177569 & 4.6   & 7.5      & F_{20 } & 3  & 287752   & 15.4    & 23.2    \\ \hline
F_5    & 2     & 2502505 & 1     & 1.5      & F_{21 } & 5  & 10117231 & 73.8    & 81.2    \\ \hline
F_6    & 2     & 2137558 & 1.6   & 2.4      & F_{22 } & 3  & 5194476  & 18.4    & 23.6  \\ \hline
F_7    & 2     & 2751159 & 3     & 4.7      & F_{23 } & 3  & 2888852  & 15.7    & 23.6  \\ \hline
F_8    & 5     & 1646099 & 13.81 & 20.5     & F_{24 } & 2  & 2901708  & 6.2     & 9.5     \\ \hline
F_9    & 2     & 1672713 & 2     & 3        & F_{25 } & 47 & 2612660  & 253.5   & 636   \\ \hline
F_{10} & 1289  & 1500768 & 117   & 234      & F_{26 } & 61 & 4257732  & 158     & 192\\ \hline
F_{11} &  2    & 2629928 & 3     & 4.1      & F_{27 } & 2  & 5925350  & 7.7     & 9.4    \\ \hline
F_{12} & 3     & 5895414 & 20    & 22.6     & F_{28 } & 3  & 5720164  & 5       & 7.1    \\ \hline
F_{13} & 11    & 1810788 & 35    & 64.6     & F_{29 } & 3  & 7826660  & 15.8    & 23   \\ \hline
F_{14} & 17    & 1618581 & 31.7  & 58       & F_{30 } & 2  & 39363539 & 14.7    & 26.3   \\ \hline
F_{15} & 2     & 7744913 & 4.9   & 6.1      & F_{31 } & 2  & 40933692 & 62.4    & 73   \\ \hline
F_{16} & 2     & 3808558 & 4.3   & 6.9      & F_{32 } & 2  & 17097775 & 82.7    & 132 \\\hline
\end{array}
$$


\end{document}